\documentclass[reqno,10pt, centertags,draft]{amsart}
\usepackage{amsmath,amsthm,amscd,amssymb,latexsym,upref}
\usepackage{graphicx,tikz}
\usepackage{latexsym}
\usepackage{lscape}
\usepackage{url}
\newcommand{\beq}{\begin{equation}}
\newcommand{\enq}{\end{equation}}

\makeatletter
\def\theequation{\@arabic\c@equation}

\newcommand{\bbR}{{\mathbb{R}}}

\newcommand{\bbZ}{{\mathbb{Z}}}
\newcommand{\bbC}{{\mathbb{C}}}

\newcommand{\bbT}{{\mathbb{T}}}

\newcommand{\cQ}{{\mathcal Q}}

\newcommand{\lb}{\label}

\newcommand{\bv}{{\mathbf v}}
\newcommand{\bk}{{\mathbf k}}

\newcommand{\bn}{{\mathbf n}}

\newcommand{\bx}{{\mathbf x}}

\newcommand{\bq}{{\mathbf q}}
\newcommand{\bp}{{\mathbf p}}
\newcommand{\bl}{{\mathbf l}}

\renewcommand{\o}{\omega}

\renewcommand{\o}{\omega}

\newcommand{\spec}{\operatorname{Spec }}

\newcommand{\bi}{\bibitem}

\renewcommand{\ln}{\text{\rm ln}}

\allowdisplaybreaks
\numberwithin{equation}{section}

\renewcommand{\det}{\operatorname{det}}

\newcommand{\curl}{\operatorname{curl}}

\renewcommand{\Re}{\operatorname{Re }}

\newcommand{\diag}{\operatorname{diag}}

\theoremstyle{plain}
\newtheorem{theorem}{Theorem}[section]

\newtheorem{lemma}[theorem]{Lemma}

\newtheorem{proposition}[theorem]{Proposition}

\theoremstyle{definition}

\newtheorem{property}[theorem]{Property}

\newtheorem{example}[theorem]{Example}

\newtheorem{remark}[theorem]{Remark}

\begin{document}
\allowdisplaybreaks

\title[Instability of unidirectional flows for 2D $\alpha$-Euler]{Instability of unidirectional flows for the 2D $\alpha$-Euler equations}
\thanks{Partially supported by  NSF grant DMS-171098, Research Council of the University of Missouri and the Simons Foundation.}
\author[H. Dullin]{Holger Dullin}
\address{School of Mathematics and Statistics,
University of Sydney NSW 2006, Australia}
\email{holger.dullin@sydney.edu.au}
\author[Y. Latushkin]{Yuri Latushkin}
\address{Department of Mathematics,
University of Missouri, Columbia, MO 65211, USA}
\email{latushkiny@missouri.edu}
%\urladdr{http://www.math.missouri.edu/personnel/faculty/latushkiny.html}
\author[R. Marangell]{Robert Marangell}
\address{School of Mathematics and Statistics,
University of Sydney NSW 2006, Australia}
\email{robert.marangell@sydney.edu.au}
\author[S. Vasudevan]{Shibi Vasudevan}
\address{International Centre for Theoretical Sciences,
Tata Institute of Fundamental Research, Bengaluru, 560089, India}
\email{shibi.vasudevan@icts.res.in}
\author[J. Worthington]{Joachim Worthington}
\address{School of Mathematics and Statistics,
University of Sydney NSW 2006, Australia}
\address{Cancer Research Division, Cancer Council NSW, Woolloomooloo, NSW 2011, Australia}
\email{joachim.worthington@nswcc.org.au}
%\email{joachimw@maths.usyd.edu.au}
\date{\today}

%\date{, 2003.}
%\keywords{Schr\"odinger equation, Hamiltonian systems, Jost function, traveling waves, eigenvalues}

\dedicatory{Dedicated to Prof.\ Tom\'as Caraballo on the occasion of his 60-th birthday}

%%%%%%%%%%%%%%%%%%%%%%%%%%%%%%%%%%%
\begin{abstract}
We study stability of unidirectional flows for the linearized 2D $\alpha$-Euler equations on the torus. The unidirectional flows are steady states whose vorticity is given by Fourier modes corresponding to a vector $\mathbf p \in \mathbb Z^{2}$. We linearize the $\alpha$-Euler equation and write the linearized operator $L_{B} $ in $\ell^{2}(\mathbb Z^{2})$ as a direct sum of one-dimensional difference operators $L_{B,\mathbf q}$ in $\ell^{2}(\mathbb Z)$ parametrized by some vectors $\mathbf q\in\mathbb Z^2$ such that the set $\{\mathbf q +n \mathbf p:n \in \mathbb Z\}$ covers the entire grid $\mathbb Z^{2}$. The set $\{\mathbf q +n \mathbf p:n \in \mathbb Z\}$ can have zero, one, or two points inside the disk of radius $\|\mathbf p\|$. We consider the case where the set $\{\mathbf q +n \mathbf p:n \in \mathbb Z\}$ has exactly one point in the open disc of radius $\mathbf p$. We show that unidirectional flows that satisfy this condition are linearly unstable. Our main result is an instability theorem that provides a necessary and sufficient condition for the existence of a positive eigenvalue to the operator $L_{B,\bq}$ in terms of equations involving certain continued fractions. Moreover, we are also able to provide a complete characterization of the corresponding eigenvector.  The proof is based on the use of continued fractions techniques expanding upon the ideas of Friedlander and Howard.
\end{abstract}
%%%%%%%%%%%%%%%%%%%%%%%%%%%%%%%%%%%

\maketitle
%{\scriptsize{\tableofcontents}}
\normalsize

%%%%%%%%%%%%%%%%%%%%%%%%%%%%%%%%%%%%
\section{Introduction and basic setup}
\subsection{Introduction} The study of eigenvalues of the differential operators obtained by linearizing the Euler and Navier Stokes equations about a steady state using the methods and techniques of continued fractions was initiated by Meshalkin and Sinai in the 1960s in their paper \cite{MS}, and since then has been pursued by many authors, for example \cite{BFY99, FH98, FSV97}. We caution the reader that this is a non exhaustive sample of the literature. See \cite{BW, BN10, WDM1, WDM2, FVY00, L} for related work on the stability of steady state solutions to the Euler equations.

In this paper we continue the work in this direction, and study stability of a special steady state, the \emph{unidirectional flow}, of the 2D $\alpha$-Euler equations on the torus written for the Fourier coefficients of vorticity. 
The $\alpha$-Euler equations are an inviscid regularization of the classical Euler equations. They were introduced and studied in a series of foundational papers by C.\ Foias, D.\ Holm, J.\ Marsden, T.\ Ratiu, E.\ Titi and others; see \cite{FHT01}, \cite{HMR98}, \cite{HMR98a} and references therein. 
The unidirectional steady state has exactly two nonzero Fourier mode corresponding to a twodimensional vector $\mathbf p \in \mathbb Z^{2}$ with integer components and its negative $-\mathbf p$.
%\footnote{JW: I have tried to make this more precise but at the cost of readability}. 
We linearize the $\alpha$-Euler equation and write the linearized operator $L_{B} $ in $\ell^{2}(\mathbb Z^{2})$ as a direct sum of one-dimensional difference operators $L_{B,\mathbf q}$ in $\ell^{2}(\mathbb Z)$ parametrized by some vectors $\mathbf q$ such that the set $\{\mathbf q +n \mathbf p:n \in \mathbb Z\}$ covers the entire grid $\mathbb Z^{2}$, see \cite{WDM1, L,LLS}. 
The set $\{\mathbf q +n \mathbf p:n \in \mathbb Z\}$ can have zero, one or two points inside the disk with radius $\|\mathbf p\|$ centred at the origin. We primarily consider the second case, and apply continued fractions to the study of spectral properties of the respective difference operator $L_{B,\mathbf q}$, cf. \cite{FH98,L,MS}. We show the existence of a positive eigenvalue for $L_{B,\mathbf q}$ in this case, which implies that $L_{B}$ has unstable spectrum. Therefore, the unidirectional steady states that have one point inside the disk of radius $\|\mathbf p\|$ are linearly unstable.
Our main result  is an instability theorem that provides a necessary and sufficient condition for the existence of a positive eigenvalue to the operator $L_{B,\bq}$ in terms of equations involving certain continued fractions. Moreover, we are also able to provide a list of additional properties of the corresponding eigenvectors.  

More details and a precise formulation are given in Theorem \ref{instability} in Section 2. Section 3 contains some results on continued fractions that are used in the proofs of the instability theorem in Section 2. In Section 4, following the ideas presented in \cite{LLS}, we characterize the essential spectrum of the linearized operator $L_{B}$ and prove the spectral mapping theorem for the group generated by $L_{B}$.

\subsection{Basic setup and governing equations}  \lb{s1}
We consider two dimensional $\alpha$-Euler equations for incompressible ideal fluid on the torus written in vorticity form,
\beq\lb{vortalpha}
\frac{\partial\omega}{\partial t}+\bv\cdot\nabla\omega=0, \quad \nabla \cdot \bv= 0, \mathbf x \in\bbT^{2},
\enq
where $\omega$ is the vorticity of the fluid and $\bv$ the smoothed velocity, $ \bv=(v_1,v_2), \bx=(x,y)\in\bbT^2=\bbR^2/2\pi\bbZ^2$.
Here 
\beq\lb{ovrel}
\omega=\curl(1-\alpha^{2}\Delta)\bv,
\enq
where $\alpha > 0$ is a positive real number.
Since $\nabla \cdot \bv =0$, there exists a stream function $\phi$, such that $\bv=-\nabla^{\perp}\phi$, where $\nabla^{\perp}=(-\partial_{y},\partial_{x})$. This means that 
\begin{equation}\lb{ophirel}
\omega= -\Delta(1-\alpha^{2}\Delta)\phi.
\end{equation}
%Assume that 
%\beq\lb{vortassump}
%\int_{\bbT^{2}} \omega dx dy =0.
%\enq
Assuming $\int_{\bbT^{2}} \omega dx dy =0$
%The fact that \eqref{vortassump} holds implies that 
allows one to solve \eqref{ophirel} for the stream function $\phi$, and in addition, by imposing the condition $\int_{\bbT^{2}} \phi dx dy =0$ one obtains a unique solution. 
%that the zero Fourier coefficient of $\phi$ satisfies $\phi_{\bf0}=0$, i.e., \[\int_{\bbT^{2}} \phi dx dy =0.\] 
%$=\frac{\partial u_2}{\partial x_1}-\frac{\partial u_1}{\partial x_2}$ is the vorticity that with no loss of generality is assumed to have zero average. 
Using the Fourier series \[\omega(\bx)=\sum_{\bk\in\bbZ^2\setminus\{0\}}\omega_\bk e^{i \bk\cdot\bx}, \quad \phi(\bx)=\sum_{\bk\in\bbZ^2\setminus\{0\}}\phi_\bk e^{i \bk\cdot\bx},\] and equation \eqref{ophirel}, one obtains the following relationship among the Fourier modes of $\omega$ and $\phi$,
\beq\lb{ophif}
\phi_{\bk}= ||\bk||^{-2}(1+\alpha^{2}||\bk||^{2})^{-1}\omega_{\bk}
\enq
for every $\bk \neq 0$. Here $||\cdot||$ denotes the standard Euclidean norm in $\mathbb R^{2}$. 
%\section{Basic Setup and Governing Equations}
Using the Fourier series expansion one can re-write the first equation in \eqref{vortalpha} for each Fourier mode $\omega_{\bk}$ of $\omega$ as 
\beq\lb{DEEalpha}
\frac{d\omega_\bk}{dt}=\sum_{\bq\in\bbZ^2\setminus\{0\}}\beta(\bk-\bq,\bq)\omega_{\bk-\bq}\omega_\bq,\,\,\bk\in\bbZ^2\setminus\{0\},
\enq
where the coefficients $\beta(\bp,\bq)$ for $\bp,\bq\in\bbZ^2$ are defined as 
\beq\lb{dfnalpha}
\beta(\bp,\bq)=\frac12\bigg(\|\bq\|^{-2}(1+\alpha^{2}\|\bq\|^{2})^{-1}-\|\bp\|^{-2}(1+\alpha^{2}\|\bp\|^{2})^{-1}\bigg)(\bp\wedge\bq)\,
\enq
for %\footnote{JW: removed case $ \bp\neq\pm\bq,$ here; not needed, already zero?} 
$\bp\neq0,\bq\neq0$, and $\beta(\bp,\bq)=0$ otherwise. Here 
\begin{equation}\label{dfnwedge}
\bp\wedge\bq=\det\left[\begin{smallmatrix}p_1&q_1\\ p_2&q_2\end{smallmatrix}\right] \mbox{ for } \bp=(p_1,p_2) \mbox{ and } \bq=(q_1,q_2).
\end{equation}
The derivation of \eqref{DEEalpha} is given in the Appendix. We refer to \cite{L} for equation \eqref{DEEalpha} in the Euler case when $\alpha=0$. 

The choice of spaces for the sequences $(\omega_\bk)_{\bk\in\bbZ^2}$ depends on the choice of vorticity in \eqref{vortalpha}. For instance, if $\omega\in H^s(\bbT^2)$, the Sobolev space, then $(\omega_\bk)\in \ell_s^2(\bbZ^2)$, the space of sequences square summable with the weight $(1+\|\bk\|^{2s})^{1/2}$. 
In what follows we will mainly consider the case $s=0$, that is, $\omega \in L^{2}(\mathbb T^{2})$ and $(\omega_{\bk}) \in \ell^{2}(\mathbb Z^{2})$ as the case $s > 0$ is analogous.

\subsection{Unidirectional flows}
A {\em unidirectional flow} is the flow induced by a time independent solution $\omega^0$ of \eqref{vortalpha} that has only one nonzero Fourier mode, that is, 
\begin{equation}\lb{dfnBS}
\omega^0(\bx)=\Re(\Gamma e^{i\bp\cdot\bx}) \text{ for a given } \bp\in\bbZ^2\setminus\{0\} \text{ and } \Gamma\in \bbC,
\end{equation}
i.e., the Fourier coefficients $\omega^{0}(\bx)$ are given by
\begin{equation}\label{steadystuni}
\omega^{0}_{\bk}=
\left\{
	\begin{array}{ll}
		\Gamma / 2& \mbox{if } \bk = \bp, \\
         \overline{\Gamma}/2 & \mbox{if } \bk = -\bp, \\	
         0 & \mbox{if } \bk \neq \pm \bp,
    \end{array}
\right.
\end{equation}
where $\bar{\Gamma}$ is the complex conjugate of $\Gamma$.

 A well-known example of  the unidirectional flow is given by the Kolmogorov  flow with vorticity $\omega^0(\bx)=\cos(mx_1)$, $m=1,2,\dots,$ (see, e.g., \cite{MS}); this corresponds to the choice $\bp=(m,0)$ and $\Gamma=1$. In the case when $m=1$ the steady state solution of the Euler equation is called in \cite{BW} a bar-state. Unidirectional flows by definition are special cases of shear flows. A shear flow has a general Fourier series but still only a flow in one direction.
 
 The unidirectional flows have been studied by many authors, see e.g. \cite{BW,WDM1,WDM2,L,LLS} and the literature therein. We demonstrate that the unidirectional flow is indeed a steady state of \eqref{DEEalpha} in Lemma \ref{steadyst} in the Appendix.

We use notation $L_B$, where $B$ stands for the ``bar state'', for the linearization of \eqref{DEEalpha} about the steady state  \eqref{dfnBS}, that is, we
linearize \eqref{DEEalpha} about the unidirectional flow and consider in $\ell^2(\bbZ^2)$ the following operator,
\begin{align}
L_B&: (\omega_\bk)_{\bk\in\bbZ^2}\mapsto
\big(\beta(\bp,\bk-\bp)\Gamma\omega_{\bk-\bp}-
\beta(\bp,\bk+\bp)\bar{\Gamma}\omega_{\bk+\bp}\big)_{\bk\in\bbZ^2}\lb{dfnLB}
\end{align}
(see the Appendix for derivation of formula \eqref{dfnLB}).
 
Our objective is to show that the spectrum of the operator $L_B$  contains an unstable eigenvalue (i.e., an eigenvalue that has a positive real part) provided $\|\bp\|$ is sufficiently large.

\subsection{Remarks} We remark that our results also pertain to the 2D Euler case by formally putting $\alpha=0$ in the $\alpha$-Euler setting. Although this paper is written for the $\alpha$-Euler equations, all the ideas, techniques and results of this current paper will carry over to the $\alpha=0$ Euler case. One can thus claim instability of unidirectional steady states for the Euler equations using the same techniques of the current paper. In other words, our results hold for every $\alpha \geq 0$. We present the results for the $\alpha$-Euler model because, despite being used in diverse areas such as turbulence modeling (see \cite{CFHOTW98, CS04}) and data assimilation (see \cite{ANT16}), very little seems to be known about the stability properties of this model. The velocity $\mathbf v$ and the vorticity $\omega$ are related via the following Biot-Savart law $\mathbf v = \nabla^{\perp}\Delta^{-1}(I-\alpha^{2}\Delta)^{-1} \omega$. Notice that the velocity is more regular in this case compared the the Euler ($\alpha=0$) case. Similar ideas involving continued fractions have been used by S.\ Friedlander and R.\ Shvydkoy, see \cite{FS05} , to characterize the unstable point spectrum of the quasi geostrophic equation which is much more singular than the present model in the sense that the Biot-Savart law relating a scalar quantity $\theta$ and the velocity $\mathbf v$ is given by $\mathbf v = \nabla^{\perp}\Delta^{-1/2} \theta$. The paper \cite{FS05} also characterizes the unstable essential spectrum of the surface quasi geostrophic equations. Furthermore, R.\ Shvydkoy, in paper \cite{S06}, has provided a characterization of the essential spectrum of a wide class of linear advective equations, examples which include, see Section 3.3 in \cite{S06}, the 2D Euler equations with and without the Coriolis rotation term, the $\alpha$-Euler equations, the surface quasi-geostrophic equations, the Boussinesq equations and the kinematic dynamo.

\section{Instability of the unidirectional flows}

In this section we first review some results regarding the operator $L_B$ defined in \eqref{dfnLB}.
We use the approach taken in \cite{WDM1,WDM2,L,LLS}. Next, we show the existence of a positive eigenvalue of $L_{B}$. Our main result is Theorem \ref{instability} proved below.

\subsection{Decomposition of subspaces and operators}\label{B-S} In this subsection we follow \cite{WDM1, L,LLS} and explain how to decompose the operator $L_B$
acting in $\ell^2(\bbZ^2)$ into the direct sum of operators $L_{B,\bq}$, $\bq\in\cQ\subset\bbZ^2$, acting in the space $\ell^2(\bbZ)$, for some set $\cQ\subset\bbZ^2$.

Let $\bp\in\bbZ^2$ be the fixed vector from \eqref{dfnBS}.
Our first objective is to construct the set $\cQ$ such that the translated vectors of the form $\bq+n\bp$, with $n\in\bbZ$ and $\bq\in\cQ$, cover the entire grid $\bbZ^2$ in a way that for different $\bq$ and $\bq'$ from $\cQ$ the sets of the translated vectors, formed by all $n\in\bbZ$, are disjoint.
To begin the construction, 
for any $\bq\in\bbZ^2$ we denote $\Sigma_{B,\bq}=\{\bq+n\bp: n\in\bbZ\}$ and note that the line $\{\bq+t\bp: t\in\bbR\}$ may contain several different sets $\Sigma_{B,\bq'}$. For a given $\bq$, we let $\tau=\tau(\bq)$ % 
temporarily denote the radius of the smallest circle centered at zero that has a nonempty intersection with the set $\Sigma_{B,\bq}$. The intersection consists of either one point (which we will denote by $\widehat{\bq}$) or two points (in this case we denote by $\widehat{\bq}$ one of them).
In other words, for each $\bq\in\bbZ^2$ we
identify the unique vector $\widehat{\bq}=\widehat{\bq}(\bq)$ in $\Sigma_{B,\bq}$ such that the following holds:
\begin{align*}
\|\widehat{\bq}\|&=\min\{\|\bq+n\bp\|: n\in\bbZ\}\text{ and }\\\widehat{\bq}&=\bq+n_{\max}\bp,\text{ where $n_{\max}=\max\{n:
 \|\bq+n\bp\|=\min\{\|\bq+n\bp\|: n\in\bbZ\}\}$.}
 \end{align*}
 The second condition simply fixes one of the possibly two points in $\Sigma_{B,\bq}$ that belong to the circle of radius $\tau=\|\widehat{\bq}\|$. We let $\cQ=\{\widehat{\bq}(\bq): \bq\in\bbZ^2\}$.

We will now decompose the operator $L_{B}$ in $\ell^2(\bbZ^2)$ into a direct sum of operators acting on the spaces isomorphic to $\ell^2(\bbZ)$. Indeed, for each $\bq\in\cQ$ we denote by $X_{B,\bq}$ the subspace of $\ell^2(\bbZ^2)$ of sequences supported in $\Sigma_{B,\bq}$, that is, we let $X_{B,\bq}=\{(\omega_\bk)_{\bk\in\bbZ^2}: \omega_\bk=0 \text{ for all $\bk\notin\Sigma_{B,\bq}\}$}$. Clearly, $\ell^2(\bbZ^2)=\oplus_{\bq\in\cQ}X_{B,\bq}$, the operator $L_B$ leaves $X_{B,\bq}$ invariant,  and therefore $L_B=\oplus_{\bq\in\cQ}L_{B,\bq}$ where $L_{B,\bq}$ is the restriction of $L_B$ onto $X_{B,\bq}$. To emphasise that $L_B$ depends on $\bp$ from \eqref{dfnBS}, we sometimes write $L_B(\bp)$ and $L_{B,\bq}(\bp)$.
For $\bk=\bq+n\bp\in\Sigma_{B,\bq}$ we denote $w_n=\omega_{\bq+n\bp}$, $n\in\bbZ$, and remark that the map $(\omega_\bk)_{\bk\in\bbZ^2}\mapsto(w_n)_{n\in\bbZ}$ is an isomorphism of $X_{B,\bq}$ onto $\ell^2(\bbZ)$. Under this isomorphism the operator $L_{B,\bq}$ in $X_{B,\bq}$ induces an operator in $\ell^2(\bbZ)$ (that we will still denote by $L_{B,\bq}$) given by the formula
\beq
L_{B,\bq}: (w_n)_{n\in\bbZ}\mapsto
\big(\beta(\bp,\bq+(n-1)\bp)\Gamma w_{n-1}-
\beta(\bp,\bq+(n+1)\bp)\bar{\Gamma} w_{n+1}\big)_{n\in\bbZ}.\lb{dfnLB1}
\enq
By \eqref{dfnalpha}, if $\bq$ is parallel to $\bp$ then $L_{B,\bq}(\bp)=0$; therefore, in what follows we will always assume that $\bq$ and $\bp$ are not parallel.

We recall that $H^{s}(\bbT^{2})$ is the Sobolev space of $2\pi$-periodic $L^{2}$ functions with $s$ derivatives in $L^{2}$. Via Fourier transform, $H^{s}(\bbT^{2})$ is isometrically isomorphic to $\ell_{s}^{2}(\bbZ^{2})$, the set of sequences $(\omega_\bk)_{\bk\in\bbZ^2}$ which are $\ell^{2}$ summable with the weight $(1+\|\bk\|^{2s})^{1/2}$. As above, we may decompose $\ell_{s}^{2}(\bbZ^{2})=\oplus_{\bq \in Q} X_{B,\bq,s}$, where $X_{B,\bq,s}$ is the space $\ell_{s}^{2}(\bbZ)$ with the weight $(1+\|\bq+n \bp\|^{2s})^{1/2}$. Since the results for $s=0$ and $s\neq0$ are analogous, in what follows we will consider only the space  $\ell^{2}(\bbZ)$.

Our objective is to study the spectrum of $L_{B,\bq}$ in $\ell^2(\bbZ)$.
From now on we assume that $\Gamma\in\bbR$. Then $L_{B,\bq}$ can be written as
$L_{B,\bq}=(S-S^*)\diag_{n\in\bbZ}\{\rho_n\}$,  where $S:(w_n)_{n\in\bbZ}\mapsto(w_{n-1})_{n\in\bbZ}$ is the shift operator in $\ell^2(\bbZ)$ and we introduce the notation
\begin{align}\lb{dfnrho}
\rho_n&=\Gamma\beta(\bp,\bq+n\bp)=\frac12\Gamma(\bq\wedge\bp)\nonumber \\  & \times \bigg(\frac{1}{\|\bp\|^{2}(1+\alpha^{2}\|\bp\|^{2})} -\frac{1}{\|\bq+n\bp\|^{2}(1+\alpha^{2}\|\bq+n\bp\|^{2})}\bigg), \, n\in\bbZ,
\end{align}
with $\mathbf q \wedge \mathbf p $ as defined in \eqref{dfnwedge}.
\begin{lemma}\label{symev}
The nonzero eigenvalues $\lambda$ of $L_{B,\bq}$ are symmetric about the coordinate axes, i.e., if $\lambda \neq 0$ is an eigenvalue, then $-\lambda, \overline{\lambda}, -\overline{\lambda}$ are also eigenvalues.
\end{lemma}
This is a result of the Hamiltonian structure of the $\alpha$-Euler equation. We refer to \cite[Prop.4, p.269]{LLS} and the Appendix for a proof.

Due to Lemma \ref{symev}, to prove spectral instability of the unidirectional flow we need to show the existence of at least one $\mathbf q \in \cQ$ such that $L_{B,\bq}$ has an eigenvalue with nonzero real part.
In turn, this is equivalent to showing that the spectrum $\spec(\frac1c L_{B,\bq})=\frac1c\spec(L_{B,\bq})$ of a multiple of $L_{B,\bq}$ has an eigenvalue with nonzero real part. Here, $c$ is any non-zero real constant that we choose. In particular,
dividing $L_{B,\bq}$ by the $n$-independent real multiple $c={\frac12\Gamma(\bq\wedge\bp)}{\|\bp\|^{-2}(1+\alpha^{2}\|\bp\|^{2})^{-1}}$, we pass to the operator $\frac1cL_{B,\bq}$ of the same structure as $L_{B,\bq}$ but with the term ${\frac12\Gamma(\bq\wedge\bp)}{\|\bp\|^{-2}(1+\alpha^{2}\|\bp\|^{2})^{-1}}$ in \eqref{dfnrho} replaced by $1$. In fact, this procedure is equivalent to rescaling $\Gamma$. In order to simplify notations we will assume in what follows that $\Gamma$  in \eqref{dfnrho} already satisfies the normalization condition \[{\frac12\Gamma(\bq\wedge\bp)}{\|\bp\|^{-2}(1+\alpha^{2}\|\bp\|^{2})^{-1}}=1.\] We introduce notation   
\begin{equation}\label{gammandef}
\gamma_{n}= - \frac{\|\bp\|^{2}(1+\alpha^{2}\|\bp\|^{2})}{\|\bq+n\bp\|^{2}(1+\alpha^{2}\|\bq+n\bp\|^{2})}.
\end{equation}
Using the normalization condition, we see that $\rho_{n}=1+\gamma_{n}$. Therefore, we want to study the spectrum of the operator 
\begin{equation}\label{lq}
L_{B,\bq}=(S-S^*)\diag_{n\in\bbZ}\{1+\gamma_n\}.
\end{equation}

\begin{remark}\label{classf}
We will now classify points $\bq\in\bbZ^2$ recalling notations $\bq$ and $\cQ$ introduced in the beginning of Subsection \ref{B-S}. For any $\bq\in\bbZ^2$ the intersection of the set $\Sigma_{B,\bq}=\{\bq+n\bp:n \in \mathbf Z\}$ with the open disc of radius $\|\bp\|$ may have either zero, one, or two points. %  
If this is the case then we call $\bq$ a point of type $0$, $I$ and $II$. 

If $\bq \in \bbZ^{2}$ is a point of type $I$ then the set $\Sigma_{B,\bq}=\{\bq+n\bp:n \in \bbZ\}$ contains exactly one vector $\hat{\bq}=\hat{\bq}(\bq)$ whose norm is stricly smaller than $\bp$. We further classify points of type $I$ as follows, see Figure 1 and Examples \ref{ex1}, \ref{ex2}, \ref{ex3}. We say that $\bq$ is of type $I_{0}$ if all other vectors in $\Sigma_{B,\bq}$ have norms strictly larger than $\|\bp\|$. This means that the only vector in $\Sigma_{B,\bq}$ whose norm does not exceed $\|\bp\|$ is located strictly inside the disk of radius $\|\bp\|$. 

There are two more possibilities for $\hat{\bq}(\bq) \in \Sigma_{B,\bq}$ to be strictly inside the disc of radius $\|\bp\|$. The first is when the preceeding point, $\hat{\bq}(\bq)-\bp$, belongs to the boundary of the disc and the second possibility is when the following point $\hat{\bq}(\bq)+\bp$ belongs to the boundary of the disc. These two cases are classified as type $I_{-}$ and $I_{+}$ respectively:
we say that $\bq$ is of type $I_{-}$ if $\|\hat{\bq}(\bq)\| < \|\bp\|$, $\|\hat{\bq}(\bq)-\bp\| = \|\bp\|$, and all other vectors in $\Sigma_{B,\bq}$ have norms strictly larger than $\|\bp\|$ and $\bq$ is of type $I_{+}$ if $\|\hat{\bq}(\bq)\| < \|\bp\|$, $\|\hat{\bq}(\bq)+\bp\| = \|\bp\|$, and all other vectors in $\Sigma_{B,\bq}$ have norms strictly larger than $\|\bp\|$.
\end{remark}
\begin{example}\label{ex1}
See Figure 1 and \cite{WDM1}. Let $\bp=(3,1)$. Then $\hat{\bq}=(-2,3)$ is of type $0$, $\hat{\bq}=(-1,2)$ is of type $I_{0}$, $\hat{\bq}=(0,-2)$ is of type $I_{+}$, $\hat{\bq}=(2,-2)$ is of type $I_{-}$ and $\hat{\bq}=(-1,1)$ is of type $II$.
\end{example}
\begin{example}\label{ex2}
Let $\bp=(1,2)$. Then $\hat{\bq}=(1,-1)$ is of type $I_{+}$, while $\hat{\bq}=(-1,1)$ is of type $I_{-}$ whereas $\hat{\bq}=(-1,0)$ is of type $II$. 
\end{example}
\begin{example}\label{ex3}
Let $\bp=(2,0)$. Then $\hat{\bq}=(0,1)$ is of type $I_{0}$. 
\end{example}

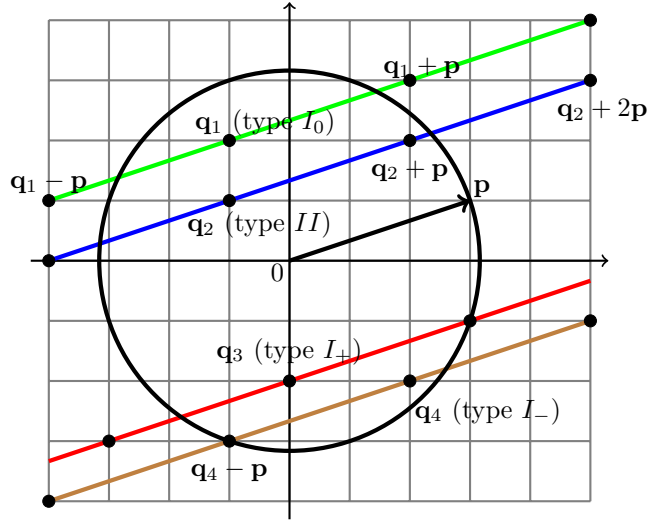
\begin{figure}\label{Classes of points}%[Classes of points]
  \begin{center}
    \begin{tikzpicture}[scale=.8]
    \draw[help lines, thick] (-4,-4) grid (5,4);
    \draw [->] [thick] (-4.3,0) -- (5.3,0);
     \draw [->] [thick] (0,-4.3) -- (0,4.3);
    \node at (-.2,-.2) {$0$};
     \draw [->] [ultra thick] (0,0) -- (3,1);
     \node at (3.2,1.2) {$\bf p$};
  %%%%%%%%%%%%%%%%%%%%%%%
     \draw [ultra thick, green] (-4,1) -- (5,4);
      \node at (-.4, 2.3) {${\bf q}_1$ (type $I_0$)};
      \node at (2.2, 3.2) {${\bf q}_1+{\bf p}$};
       \node at (-4, 1.3) {${\bf q}_1-{\bf p}$};
      \draw [fill] (-1,2) circle [radius=.1];
     \draw [fill] (2,3) circle [radius=.1];
      \draw [fill] (-4,1) circle [radius=.1];
       \draw [fill] (5,4) circle [radius=.1];
      %%%%%%%%%%%%%%%%%%%%%
      \draw [ultra thick, blue] (-4,0) -- (5,3);
       \draw [fill] (-1,1) circle [radius=.1];
     \draw [fill] (2,2) circle [radius=.1];
      \draw [fill] (-4,0) circle [radius=.1];
          \draw [fill] (5,3) circle [radius=.1];
       \node at (-.5, .6) {${\bf q}_2$ (type $II$)};
        \node at (2, 1.5) {${\bf q}_2+{\bf p}$};
         \node at (5.2, 2.5) {${\bf q}_2+2{\bf p}$};
       %%%%%%%%%%%%%%%%%%%%%%%%%%
       \draw [ultra thick, red] (-4,-3.3333) -- (5,-.3333);
        \draw [fill] (0,-2) circle [radius=.1];
        \draw [fill] (3,-1) circle [radius=.1];
         \node at (0, -1.5) {${\bf q}_3$ (type $I_+$)};
     %%%%%%%%%%%%%%%%%%%%%%%%%%
     \draw [ultra thick, brown] (-4,-4) -- (5,-1);
      \node at (3.3, -2.5) {${\bf q}_4$ (type $I_-$)};
      \node at (-1, -3.5) {${\bf q}_4-{\bf p}$};
       \draw [fill] (-1,-3) circle [radius=.1];
        \draw [fill] (2,-2) circle [radius=.1];
         \draw [fill] (5,-1) circle [radius=.1];
          \draw [fill] (-4,-4) circle [radius=.1];
        \draw [fill] (-3,-3) circle [radius=.1];
     \draw[ultra thick] (0,0) circle [radius=3.1622];
\end{tikzpicture}
\caption{${\bf p}=(3,1)$; point ${\bf q}_1=(-1,2)$ is a point of type $I_0$ (green $\Sigma_{{\bf q}_1}$), point ${\bf q}_2=(-1,1)$ is a point of type $II$ (blue $\Sigma_{{\bf q}_2}$), point ${\bf q}_3=(0,-2)$ is a point of type $I_+$ (red $\Sigma_{{\bf q}_3}$), and point ${\bf q}_4=(2,-2)$ is a point of type $I_-$ (brown $\Sigma_{{\bf q}_4}$ ).}
\end{center}
\end{figure}
In what follows, dealing with the operator $L_{B,\bq}$ from \eqref{dfnLB1}, %for a point $\bq$ of type $I$, 
we will drop hat in the notation $\hat{\bq}$, that is, we assume that $\bq \in \bbZ^{2}$ satisfies $\|\bq\| < \|\bp\|$. 
\begin{remark}\label{remlist} The fact that $\bq$ is a point of type $0$, $I$, or $II$ leads to the following respective conclusions:

{\bf (i)}\, Assume that $\|\bq\| \geq \|\bp\|$, that is, $\bq$ is a point of type $0$. Since $\bq \in \cQ$ is chosen to minimize $\|\bq+n\bp\|$, we know that $\|\bq+n\bp\| \geq \|\bp\|$ and therefore $|\gamma_{n}| \leq 1$ or $1+\gamma_{n} \geq 0$ for all $n \in \mathbb Z$.

{\bf (ii)}\, Assume that $\|\bq\| < \|\bp\|$ and that the line $\Sigma_{B,\bq}$ has exactly one point in the open disc of radius $\|\bp\|$ (that is, we assume that $\bq$ is a point of type $I$). Then ${(1+\alpha^{2}\|\bp\|^{2}) }>{ (1+\alpha^{2}\|\bq\|^{2})}$. 
If $\bq$ is of type $I_{0}$ then $\rho_{0}<0$ and $\rho_{n}=1+\gamma_{n}>0$ for all $n \neq 0$. If $\bq$ is of type $I_{+}$, then $\rho_{0}<0$ and $\rho_{1}=1+\gamma_{1}=0$ and $\rho_{n}=1+\gamma_{n}>0$ for all $n \neq 0,1$. If $\bq$ is of type $I_{-}$, then $\rho_{0}<0$ and $\rho_{-1}=1+\gamma_{1}=0$ and $\rho_{n}=1+\gamma_{n}>0$ for all $n \neq 0,-1$. 

{\bf (iii)}\, Assume that $\bq$ is a point of type $II$, i.e., we assume that  ${\|\bp\| }>{ \|\bq\|} $,
that ${\|\bp\| }>{ \|\bq-\bp\|}$,  and that ${\|\bp\| }\le{ \|\bq+n\bp\|}$ for all $n \in \mathbb Z \backslash \{0,-1\}$. Then $1+\gamma_{0}<0$, $1+\gamma_{-1}< 0$ but $1+\gamma_{n} \geq 0$ for all $n \in \mathbb Z \backslash \{0,-1\}$. 
\end{remark}

The operator $L_{B,\bq}$ defined in \eqref{lq} is a product of two operators and can be viewed as an infinite matrix with two nonzero diagonals. It is sometimes convenient to make this matrix more symmetric by putting a square root of the operator $\diag_{n\in\bbZ}\{1+\gamma_n\}$ in front of the multiple $S-S^*$. To achieve that, using \eqref{gammandef}, we introduce the following notation,
\begin{equation}\label{dn}
\delta_{n} = 
\left\{
	\begin{array}{ll}
		\sqrt{1+\gamma_{n}}  & \mbox{for } 1+\gamma_{n} \geq 0, \mbox{ when } \delta_{n} \in \mathbb R,\\
		i \sqrt{|1+\gamma_{n}|} & \mbox{for } 1+\gamma_{n} < 0, \mbox{ when } \delta_{n} \in i \mathbb R,
	\end{array}
\right.
\end{equation}
so that $\delta_{n}^{2}=1+\gamma_{n}$. Since $L_{B,\bq}=(S-S^{*})\diag_{n\in\bbZ}\{\delta_{n}\}\diag_{n\in\bbZ}\{\delta_{n}\}$, the nonzero elements of the spectrum of $L_{B,\bq}$ coincide with the nonzero elements of the spectrum of the operator $M_{\bq}$ defined by
\begin{equation}\label{mq}
M_{\bq}=\diag_{n\in\bbZ}\{\delta_{n}\}(S-S^{*})\diag_{n\in\bbZ}\{\delta_{n}\}.
\end{equation}
This is a consequence of the following well-known fact:

\begin{lemma}\label{spectrumcommute}
Suppose $A,B:X \to X$ are bounded linear operators on a Banach space $X$. Then $\sigma(AB)\backslash \{0\} = \sigma(BA)\backslash \{0\}$.
\end{lemma} 
We can thus study the spectrum of the operator $M_{\bq}$ instead of $L_{B,\bq}$.
The operator $M_{\bq}$ has the following structure:
\begin{equation*}
 M_\bq=
\left[ \begin{array}{ccccccc}  \ddots & & & &  \\
 &0& -\delta_{-2}\delta_{-1}&0& 0&0  & \\
& \delta_{-2}\delta_{-1}&0& -\delta_{-1}\delta_{0} &0&0 &\\
& 0 & \delta_{-1}\delta_0 & \text{\fbox{$0$}} & - \delta_{0}\delta_{1} &0 &\\ 
& 0 & 0 & \delta_{0}\delta_{1} & 0 & - \delta_{1}\delta_{2} & \\ 
& 0 & 0 & 0 & \delta_{1}\delta_{2} & 0 &\\
& & & & & & \ddots  %& \overline{\alpha }\delta_{-1}\delta_0 & \text{\fbox{$0$}} & \alpha \delta_0\delta_1 & \\ 
%& 0 & \overline{\alpha} \delta_0\delta_1 & 0 & \ddots
\end{array}\right].
\end{equation*}
The ``central'' entry has been marked with a box, for future reference. 
We remark that $\delta_{n} \to 1$ and $n \to \infty$ since $\gamma_{n} \to 0$ and that $M_{\bq}$ is a compact perturbation of $S-S^{*}$, therefore $\spec_{\rm ess}(M_{\bq})=\spec(S-S^{*})=i[-2,2]$. 

If $\bq$ is a point of type $0$ then $L_{B,\bq}$ has no unstable point spectrum (cf.\ \cite[Remark 4]{LLS}). Indeed, if $\delta_{n} \in \mathbb R$ for all $n$, i.e., $\bq$ is a point of type $0$ and $\|\bq\| \geq \|\bp\|$, then $M_{\bq}^{*}=-M_{\bq}$, i.e., $M_{\bq}$ is skew-adjoint and its spectrum is thus purely imaginary. 
%\subsection{}

We now consider $M_{\bq}$ for $\bq$ being of type $I$ or $II$. Then two cases are possible:
\begin{enumerate}
\item[(a)] $\delta_{0} \in i \bbR$ and $\delta_{n} \in \bbR$ for all $n \neq 0$;
\item[(b)] $\delta_{0},\delta_{-1} \in i \bbR$ and $\delta_{n} \in \bbR$ for all $n \neq 0,-1$.
\end{enumerate}
We note that case (a) corresponds to item (ii) while case (b) corresponds to item (iii)
in the list given in Remark \ref{remlist}.

In case (a) the $3\times 3$ block
\begin{equation*}
\left[ \begin{array}{ccccc}  
 &0& -\delta_{-1}\delta_{0}&0&  \\
& \delta_{-1}\delta_{0}&\text{\fbox{$0$}} & -\delta_{0}\delta_{1} &\\
& 0 & \delta_{0}\delta_{1} & 0 &  
\end{array}\right]
\end{equation*} 
is self adjoint while the remaining part of $M_{\bq}$ is skew-adjoint because $\delta_{l-1}\delta_{l} \in i \bbR$ only for $l=0,1$ and $\delta_{i-1}\delta_{l} \in \bbR$ for $l \neq 0,1$.
In case (b) we have $\delta_{0}, \delta_{-1} \in i \bbR$ and $\delta_{n} \in \bbR$ for $n \neq 0,-1$ and then $\delta_{l-1}\delta_{l} \in i \bbR$ provided that $l=-1,1$ and $\delta_{l-1}\delta_{l} \in \bbR$ for $l \neq -1,1$. This means that in case (a) or (b) we do not know that the spectrum of $M_{\bq}$ is purely imaginary and there is a possibility that unstable eigenvalues exist.

 Indeed, if $\bq$ is a point of type $I$ then the arguments given in Subsection \ref{ss-cf} (cf.\ also \cite{WDM1}) based on the use of continued fractions yield the existence of an unstable eigenvalue for $L_{B,\bq}$. In a sense, we adapt to the current setting the proof from \cite{FH98} used therein for the Orr-Sommerfeld operator, see also \cite{MS}. However, if $\bq$ is a point of type $II$ then the question whether or not there are unstable (complex) eigenvalues is an important open problem.
%%%%%%%%%%%%%%%%%%%%%%%%%%%%%%%%%%%%%%%%%%%%%%
\subsection{Unstable eigenvalues for unidirectional flows in case of the point of type $I$}\lb{ss-cf}
The main result of this subsection states that the linearized Euler operator $L_B$ has a positive eigenvalue provided at least one point $\bq\in\cQ(\bp)$ is of type $I$. Here, we are using the classification of points given in the previous subsection, see Remark \ref{classf}. Specifically, we will show
that if $\bq$ is the only point in $\Sigma_{B,\bq}=\{\bq+n\bp: n\in\bbZ\}$ satisfying $\|\bq\|<\|\bp\|$, i.e. if $\bq$ is of type $I$, then $L_{B,\bq}$ has a positive eigenvalue. 
We recall that by \eqref{gammandef} the coefficients in $L_{B,\bq}$ from \eqref{lq} are given by the formula
\begin{equation}\label{rhonexp}
\rho_{n}=1+\gamma_{n}=1-\frac{\|\bp\|^{2}(1+\alpha^{2}\|\bp\|^{2})}{\|\bq+n \bp\|^{2}(1+\alpha^{2}\|\bq+n \bp\|^{2})}, \quad n \in \bbZ.
\end{equation}
 %\[ \rho_{n}=1+\gamma_{n}=1-\frac{\|\bp\|^{2}(1+\alpha^{2}\|\bp\|^{2})}{\|\bq+n \bp\|^{2}(1+\alpha^{2}\|\bq+n \bp\|^{2})}.\] 

For simplicity, we first consider a point $\bq$ of type $I_{0}$, and outline an informal argument that shows the existence of a positive eigenvalue of $L_{B,\bq}$. In this case $\|\bq\| < \|\bp\|$ and $\|\bq+n\bp\| > \|\bp\|$ for all $n \neq 0$. That is, $-1< \gamma_{n}<0$ for all $n \neq 0$ and $\gamma_{0}<-1$.
%We recall the convention  $\frac{\frac12\Gamma(\bq\wedge\bp)}{\|\bp\|^{2}(1+\alpha^{2}\|\bp\|^{2})}=1$ in \eqref{dfnrho} 
This implies that if the point $\bq$ is of type $I_{0}$ then
\beq\lb{rhoineq}
\text{$\rho_0<0$ and $\rho_n>0$ for all $n\neq0$.}
\enq 
We consider the eigenvalue problem
\beq\label{evalproblem}
L_{B,\bq}(w_n)_{n\in\bbZ}=\lambda(w_n)_{n\in\bbZ}.
\enq
Letting $z_n=\rho_n w_n$, equation \eqref{evalproblem} is equivalent to the difference equation 
\beq\lb{diffeq1}
z_{n-1}-z_{n+1}=\frac{\lambda}{\rho_n}z_n,\,n\in\bbZ,
%\, \rho_n=1-\frac{\|\bp\|^{2}(1+\alpha^{2}\|\bp\|^{2})}{\|\bq+n \bp\|^{2}(1+\alpha^{2}\|\bq+n \bp\|^{2})},\, n\in\bbZ,
\enq
where $\rho_{n}$ are given by formula \eqref{rhonexp}.
%provided , where we have denoted $z_n=\rho_nw_n$. 
Note that $\rho_n\to1$ as $|n|\to\infty$. Assuming $w_{n} \neq 0$ for any $n$, we
%Assuming $w_{n} \neq 0$ for every $n \in \mathbb Z$, we 
introduce the notation $u_n=z_{n-1}/z_n$ (and note that $z_{n} \neq 0$ for any $n$ since $w_{n} \neq 0$), and re-write \eqref{diffeq1} as
\beq\label{iterate}
u_n=\frac{\lambda}{\rho_n}+\frac{1}{u_{n+1}}\quad \text{ or }\quad
u_{n+1}=-\frac{1}{\frac{\lambda}{\rho_n}-u_n}, \quad n \in \bbZ.
\enq
Forwards iterating the first equation in \eqref{iterate} for $n \geq 0$ and backwards iterating the second equation for $n \leq -1$, we obtain two $\lambda$-depending sequences,
\begin{align}\lb{uone}
u_n^{(1)}(\lambda)&=\frac{\lambda}{\rho_n}+
\cfrac{1}{
\frac{\lambda}{\rho_{n+1}}+\cfrac{1}{
\frac{\lambda}{\rho_{n+2}}+\dots}}, \, n=0,1,2,\dots,\\
u_{n+1}^{(2)}(\lambda)&=-\cfrac{1}{
\frac{\lambda}{\rho_{n}}+\cfrac{1}{
\frac{\lambda}{\rho_{n-1}}+\cfrac{1}{\frac{\lambda}{\rho_{n-2}}+\dots}}}, \, n=-1,-2,\dots,
\lb{utwo}
\end{align}
from which we obtain the following two formulas for the entry $u_0=u_0(\lambda)$ of the solution $(u_n)$ to the difference equation \eqref{iterate}:
\begin{align*}
u_0^{(1)}(\lambda)&=\frac{\lambda}{\rho_0}+\frac1{u_1}=
\frac{\lambda}{\rho_0}+\frac1{\frac{\lambda}{\rho_1}+\frac1{u_2}}=\dots=\frac{\lambda}{\rho_0}+f(\lambda),\\
u_0^{(2)}(\lambda)&=-\frac1{\frac{\lambda}{\rho_{-1}}-\frac1{u_{-1}}}=
-\frac1{\frac{\lambda}{\rho_{-1}}-\frac1{\frac{\lambda}{\rho_{-2}}-\frac1{u_{-2}}}}=\dots=-g(\lambda),
\end{align*}
where we introduce $f(\lambda)$ and $g(\lambda)$ as the continued fractions
\begin{equation}\label{fgdef}
f(\lambda)=\cfrac{1}{\frac{\lambda}{\rho_1}+\cfrac{1}{\frac{\lambda}{\rho_2}+\cdots}}\,,\quad
g(\lambda)=\cfrac{1}{\frac{\lambda}{\rho_{-1}}+\cfrac{1}{\frac{\lambda}{\rho_{-2}}+\cdots}}.
\end{equation}
We refer to Section 3 for  basic results concerning continued fractions. The continued fractions in \eqref{fgdef} converge by the Van Vleck Theorem, see \cite[Theorem 4.29]{JT}. 

Clearly (as we prove in Lemma \ref{lambda-lemma}(1) below), $\lambda>0$ is an eigenvalue of $L_{B,\bq}$ with an eigenvector $(w_{n})$ provided there is a corresponding solution $(u_n)$ to \eqref{iterate} which, in turn, happens if and only if $u_0^{(1)}(\lambda)=u_0^{(2)}(\lambda)$, or, equivalently,
if and only if $\lambda$ satisfies  the equation
\beq\lb{eqnlambda}
\frac{\lambda}{\rho_0}+f(\lambda)+g(\lambda)=0.
\enq
Thus, to show the existence of a positive eigenvalue of $L_{B, \bq}$ it is enough to show the existence of  a positive root of equation \eqref{eqnlambda}. 

Using \eqref{rhoineq} we observe that if $\bq$ is of type $I_{0}$ then both functions $f$ and $g$ take positive values for positive $\lambda$. We will also see in Lemma \ref{u-lemma}(4) that
\begin{align}\label{fglimits}
\lim_{\lambda\to0^+}f(\lambda)=\lim_{\lambda\to0^+}g(\lambda)=1,\, \lim_{\lambda\to+\infty}f(\lambda)=\lim_{\lambda\to+\infty}g(\lambda)=0.
\end{align}
Since $\rho_0<0$ by \eqref{rhoineq}, equation \eqref{eqnlambda} must have a positive root, as claimed. A similar argument works if $\bq$ is of type $I_{-}$, that is, $\rho_{1} \neq 0$ and $\rho_{-1}=0$. In this case we will use $f(\lambda)$ as in \eqref{fgdef} and set $g(\lambda)=0$ in \eqref{eqnlambda}. If $\bq$ is of type $I_{+}$, that is, $\rho_{1}=0$ and $\rho_{-1} \neq 0$, we will use $g(\lambda)$ as in \eqref{fgdef} and set $f(\lambda)=0$ in \eqref{eqnlambda}.

We will show below that condition \eqref{eqnlambda} is not only sufficient but is also necessary for $\lambda$ to be an eigenvalue of the operator $L_{B,\bq}$. Since the respective eigensequence $(w_n)$ is related to the sequence $(u_n)$ from \eqref{iterate}, and the latter is eventually given by means of the continued fractions in equations \eqref{uone} and \eqref{utwo}, where, by construction, $u_{n}^{(1)}(\lambda)>0$ for $n >0$ and $u_{n}^{(2)}(\lambda)<0$ for $n \leq 0$, the sequence $(w_n)$ must possess some additional properties. Indeed, due to \eqref{uone} and \eqref{utwo}, we require our $w_{n}$ to be such that $u_{n} > 0$ for $n \geq 1$ and $u_{n} < 0$ for $n \leq 0$.
 Using the formulas $z_{n}=\rho_{n}w_{n}$ and $u_{n}=z_{n-1}/z_{n}$ one can check directly that either one of the following two possibilities must happen: Either (a): $w_{n}$ must be so that $w_{n} > 0$ for $n \geq 1$, $w_{0} < 0$, $w_{-1} < 0$ and $w_{-2}, w_{-4},\ldots $ are all positive while $w_{-1},w_{-3},\ldots$ are all negative; or (b): the sequence  $(-1)w_{n}$ satisfies these inequalities.% We record the property of the eigenvector $(w_{n})$ in Property \ref{eigvectorio} below.

We will now proceed with a more formal proof of the fact that if $\bq$ is a point of type $I$ then $L_{B,\bq}$ has a positive eigenvalue with 
the eigenvector $(w_{n})$ satisfying 
\begin{property}\label{eigvectorio}
\noindent
\begin{enumerate}
\item[(1)]In case $\bq$ is of type $I_{0}$, the eigenvector $(w_{n})$ of \eqref{evalproblem} is such that the following holds: either $w_{n}>0$ for $n>0$, $w_{n}<0$ for $n=-1,0$, and $(-1)^{|n|}w_{n}>0$ for $n \leq -2$, or the entries of the vector $(-w_{n})$ satisfy the inequalities just listed. 
\item[(2)]In case $\bq$ is of type $I_{+}$, the eigenvector $(w_{n})$ of \eqref{evalproblem} is such that the following holds: either $w_{n}=0$ for $n>1$, $w_{1}>0$, $w_{n}<0$ for $n=-1,0$, and $(-1)^{|n|}w_{n}>0$ for $n \leq -2$, or the entries of the vector $(-w_{n})$ satisfy the inequalities just listed. 
\item[(3)] In case $\bq$ is of type $I_{-}$, the eigenvector $(w_{n})$ of \eqref{evalproblem} is such that the following holds: either $w_{n}=0$ for $n< -1$, $w_{n}<0$ for $n=-1,0$, and $w_{n}>0$ for $n > 0$, or the entries of the vector $(-w_{n})$ satisfy the inequalities just listed. \end{enumerate}
\end{property}
Thus, if $\bq$ is of type $I_{0}$ and Property \ref{eigvectorio} holds  then the entries $w_{n}$ are of alternating signs if $n <0$, that is, $w_{-1}, w_{-3}, w_{-5},\ldots$ are all negative and $w_{-2}, w_{-4}, w_{-6},\ldots$ are all positive, and, in particular, $w_{n} \neq 0$ for any integer $n$.
If $\bq$ is of type $I_{+}$ and Property \ref{eigvectorio} holds then the entries of the eigenvector $(w_{n})$ satisfy the same inequalities as the case when $\bq$ is of type $I_{0}$ except that $w_{n}=0$ for $n >1$.
If $\bq$ is of type $I_{-}$ and Property \ref{eigvectorio} holds then the entries of the eigenvector $(w_{n})$ satisfy the same inequalities as the case when $\bq$ is of type $I_{0}$ except that $w_{n}=0$ for $n <-1$.

We recall the notation for the weighted spaces $\ell^{2}_{s}(\bbZ^{2})$ and $\ell_{s}^{2}(\bbZ)$ given in the discussion following \eqref{dfnLB1}.
Our main theorem is the following. 
\begin{theorem}\label{instability}
Assume that $\bp \in \bbZ^{2}$ is such that at least one point $\bq \in \cQ(\bp)$ is of type $I$, where $\bq$ is not parallel to $\bp$. Also, we assume that $\Gamma \in \mathbb R$ and satisfies the normalization condition \[{\frac12\Gamma(\bq\wedge\bp)}{\|\bp\|^{-2}(1+\alpha^{2}\|\bp\|^{2})^{-1}}=1.\] Then the steady state $(\o_{\bk}^{0})_{\bk \in \bbZ^{2}\backslash \{0\}}$ defined in \eqref{steadystuni} is linearly unstable. 

In particular, the operator $L_{B,\bq}$ in the space $\ell^{2}_{s}(\bbZ)$ has a positive eigenvalue and therefore $L_{B}$ in $\ell_{s}^{2}(\bbZ^{2})$ has a positive eigenvalue. 

Moreover, the following assertions  hold. 
\begin{enumerate}
\item[(1)] If $\bq$ is of type $I_0$ then $\lambda > 0$ is an eigenvalue of $L_{B,\bq}$ with eigenvector $(w_{n})$ satisfying Property \ref{eigvectorio}(1)
if and only if $\lambda > 0$ is a solution to the equation
\begin{equation}\label{eqio}
\frac{\lambda}{\rho_{0}}+f(\lambda)+g(\lambda)=0.
\end{equation}
\item[(2)] If $\bq$ is of type $I_+$ then $\lambda > 0$ is an eigenvalue of $L_{B,\bq}$ with eigenvector $(w_{n})$ satisfying Property \ref{eigvectorio}(2) if and only if $\lambda > 0$ is a solution to the equation
\begin{equation}\label{eqip}
\frac{\lambda}{\rho_{0}}+g(\lambda)=0.
\end{equation}
\item[(3)] If $\bq$ is of type $I_-$ then $\lambda > 0$ is an eigenvalue of $L_{B,\bq}$ with eigenvector $(w_{n})$ satisfying Property \ref{eigvectorio}(3) if and only if $\lambda > 0$ is a solution to the equation
\begin{equation}\label{eqim}
\frac{\lambda}{\rho_{0}}+f(\lambda)=0.
\end{equation}
\end{enumerate}
\end{theorem}
%First, we consider the case when $\bq$ is of type $I_{0}$.
Before presenting the proof of Theorem \ref{instability} we will need two lemmas. Their proofs rely on the auxiliary material on continued fractions contained in Section \ref{contfracresults}.

\begin{lemma}\label{u-lemma}
Assume $\bq$ is of type $I_{0}$, fix any positive $\lambda$ and consider the following continued fractions, 
\begin{equation}\label{un1}
u_{n}^{(1)}(\lambda):=\frac{\lambda}{\rho_{n}} + \Big[ \frac{\lambda}{\rho_{n+1}}, \ldots \Big]=\frac{\lambda}{\rho_{n}}+\cfrac{1}{\frac{\lambda}{\rho_{n+1}}+\cfrac{1}{\frac{\lambda}{\rho_{n+2}}+\cdots}}\,,\, n=0,1,2,\ldots,
\end{equation}
\begin{equation}\label{un2}
u_{n+1}^{(2)}(\lambda):=-\Big[\frac{\lambda}{\rho_{n}}, \frac{\lambda}{\rho_{n+1}}, \ldots \Big]=-\cfrac{1}{\frac{\lambda}{\rho_{n}}+\cfrac{1}{\frac{\lambda}{\rho_{n-1}}+\cdots}}\,,\, n=-1,-2,\ldots .
\end{equation}
Then the following assertions hold:
\begin{enumerate}
\item[(1)] $u_{n}^{(1)}(\lambda)$ and $u_{n}^{(2)}(\lambda)$ are convergent continued fractions and the functions $u_{n}^{(1)}(\cdot)$ and $u_{n}^{(2)}(\cdot)$ are continuous in $\lambda$. 
\item[(2)] There exist limits
\begin{equation*}
u_{\infty}^{(1)}(\lambda)=\lim_{n \to \infty} u_{n}^{(1)}(\lambda), \quad u_{-\infty}^{(2)}(\lambda)=\lim_{n \to -\infty} u_{n}^{(2)}(\lambda), \quad \lambda >0,
\end{equation*}
satisfying $|u_{\infty}^{(1)}(\lambda)| > 1$, $|u_{-\infty}^{(2)}(\lambda)| <1$.
\item[(3)] For some $0<q < 1$ and $C>0$, the following hold 
\begin{equation}\label{u1ineq}
(|u_{1}^{(1)}(\lambda)u_{2}^{(1)}(\lambda)\ldots u_{n}^{(1)}(\lambda)|)^{-1} \leq Cq^{n},  \mbox{ for all }n \geq 0,
\end{equation}
\begin{equation}\label{u2ineq}
(|u_{n}^{(2)}(\lambda)\ldots u_{-2}^{(2)}(\lambda)u_{-1}^{(2)}(\lambda)u_{0}^{(2)}(\lambda)|) \leq Cq^{-n}, \mbox{ for all } n \leq -1.
\end{equation}
\item[(4)] $\lim_{\lambda \to 0^{+}} |u_{0}^{(k)}(\lambda)|=1$, $\lim_{\lambda \to +\infty} u_{0}^{(k)}(\lambda)=0$ for $k=1,2$.
\end{enumerate}
\end{lemma}

\begin{proof}
(1) This follows from the Van Vleck theorem and the Stjeltjes-Vitali Theorem, see \cite[Theorem 4.29 and Theorem 4.30]{JT}, since $\lambda > 0$, and thus $\arg \lambda$ satisfies $|\arg\lambda| < \frac{\pi}{2}-\varepsilon$ and hence the continued fractions $u_{n}^{(1)}(\lambda)$ and $u_{n}^{(2)}(\lambda)$ converge. In addition, the Van Vleck Theorem also guarantees that the maps $\lambda \mapsto u_{n}^{(1)}(\lambda),u_{n}^{(2)}(\lambda)$ are holomorphic in $\lambda$ since $|\arg\lambda| \leq \frac{\pi}{2}-\varepsilon$ implying the continuity clause. 

(2) The fact that the limits $u_{\infty}^{(1)}(\lambda)$ and $u_{-\infty}^{(2)}(\lambda)$ exist follows from item (3) in Lemma \ref{cf} proved in Section \ref{contfracresults}. Passing to the limit as $n \to \infty$ in \eqref{un1} and \eqref{un2} we see that \[u_{\infty}^{(1)}(\lambda)=\lambda + 1 /u_{\infty}^{(1)}(\lambda)  \mbox{ and } u_{-\infty}^{(2)}(\lambda)=\frac{-1}{\lambda - u_{-\infty}^{(2)}(\lambda)}\] since $\rho_{n} \to 1$ as $n \to \infty$. Thus, we notice that both $u_{\infty}^{(1)}$ and  $u_{-\infty}^{(2)}$ satisfy the following quadratic equation
\[ u_{\pm \infty}^{2}-\lambda u_{\pm \infty}-1=0,\] the solutions of which are given by $u_{\pm \infty}= (\lambda/2) \pm(({\lambda}/{2})^{2}+1)^{1/2}$. Notice also that $u_{\infty}^{(1)}(\lambda)$ must be positive and  $u_{-\infty}^{(2)}(\lambda)$ must be negative. From these it is seen that $u_{\infty}^{(1)} = ({\lambda}/{2}) +(({\lambda}/{2})^{2}+1)^{1/2}$ and $u_{-\infty}^{(2)} = ({\lambda}/{2}) - (({\lambda}/{2})^{2}+1)^{1/2}$ and $|u_{\infty}^{(1)}(\lambda)| > 1$, $|u_{-\infty}^{(2)}(\lambda)| <1$.

(3) Let $q' \in (1,u_{\infty}^{(1)}(\lambda))$. Note that from (2), since $u_{\infty}^{(1)}(\lambda) > 1$, there exists an integer $N_{q'}$ such that if $n > N_{q'}$, then $u_{n}^{(1)}(\lambda) > q'$. We thus have that,
\begin{align*} 
 u^{(1)}_{1}(\lambda)u^{(1)}_{2}(\lambda)\cdots u^{(1)}_{n}(\lambda)& = u^{(1)}_{1}(\lambda)\cdots u^{(1)}_{N_{q'}}(\lambda)u^{(1)}_{N_{q'}+1}(\lambda)\ldots u^{(1)}_{n}(\lambda) \\ &\geq u^{(1)}_{1}(\lambda)\cdots u^{(1)}_{N_{q'}}(\lambda) {q'}^{n-N_{q'}} 
 %=\frac{u^{(1)}_{1}(\lambda)\cdots u^{(1)}_{N_{q'}}(\lambda)}{{q'}^{N_{q'}}}{q'}^{n}
  = \frac{1}{C} {q'}^{n},
\end{align*}
where we have denoted  $C=C(q')=\big({u^{(1)}_{1}(\lambda)\ldots u^{(1)}_{N_{q'}}(\lambda)}{q'^{-N_{q'}}}\big)^{-1}$. Let $q=1/q'$ and we thus obtain \eqref{u1ineq}.
Since $|u_{-\infty}^{(2)}(\lambda)| <1$, we have that for a fixed $q$ such that $ |u_{-\infty}^{(2)}(\lambda)| < q <1$, there exists an integer $N_{q}>0$ such that if $n < -N_{q}$, then $|u_{n}^{(2)}(\lambda)| < q$. We thus have that,
\begin{align*} 
 |u^{(2)}_{0}(\lambda)u^{(2)}_{-1}(\lambda)\cdots u^{(2)}_{n}(\lambda) | &= |u^{(2)}_{0}(\lambda)u^{(2)}_{-1}(\lambda)\cdots u^{(2)}_{-N_{q}}(\lambda)u^{(2)}_{-N_{q}-1}(\lambda)\ldots u^{(2)}_{n}(\lambda) | \\ & \leq |u^{(2)}_{0}(\lambda)u^{(2)}_{-1}(\lambda)\ldots u^{(2)}_{-N_{q}}(\lambda)| {q}^{n-N_{q}} 
 %=\frac{|u^{(2)}_{0}(\lambda)u^{(2)}_{-1}(\lambda)\cdots u^{(2)}_{-N_{q}}(\lambda)|}{{q}^{N_{q}}}{q}^{n} 
 = C {q}^{n},
\end{align*}
where we have denoted  $C=C(q)={|u^{(2)}_{0}(\lambda)u^{(2)}_{-1}(\lambda)\cdots u^{(2)}_{-N_{q}}(\lambda)|}{q^{-N_{q}}}$. This proves \eqref{u2ineq}.

(4) Noticing that $u_{0}^{(1)}(\lambda)=\lambda/\rho_{0}+f(\lambda)$ and $u_{0}^{(2)}(\lambda)=-g(\lambda)$, this follows from items (4) and (5) in Lemma \ref{cf} proved in Section \ref{contfracresults}.
\end{proof}
\begin{remark}
If $\bq$ is of type $I_{+}$, we will use the continued fraction $u_{n}^{(2)}(\lambda)$ for $n \leq 0$ and if $\bq$ is of type $I_{-}$, we will use the continued fraction $u_{n}^{(1)}(\lambda)$ for $n \geq 0$.
\end{remark}

\begin{lemma}\label{lambda-lemma}
Fix any positive $\lambda > 0$ and consider the continued fractions $u_{n}^{(1)}(\lambda)$ and $u_{n}^{(2)}(\lambda)$ given in \eqref{un1} and \eqref{un2}. Then the following hold.
%If $|arg\lambda| \leq \frac{\pi}{2}-\varepsilon$ and $\lambda \neq 0$ then hav
\begin{enumerate}
\item[(1)] If $\bq$ is of type $I_{0}$, then $\lambda \in \sigma_{disc}(L_{B,\bq})$ with eigenvector $(w_{n})$ satisfying Property \ref{eigvectorio}(1) if and only if $u_{0}^{(1)}(\lambda)=u_{0}^{(2)}(\lambda)$.
\item[(1P)] If $\bq$ is of type $I_{+}$, then $\lambda \in \sigma_{disc}(L_{B,\bq})$ with eigenvector $(w_{n})$ satisfying Property \ref{eigvectorio}(2) if and only if $u_{0}^{(2)}(\lambda)=\lambda / \rho_{0}$.
\item[(1M)] If $\bq$ is of type $I_{-}$, then $\lambda \in \sigma_{disc}(L_{B,\bq})$ with eigenvector $(w_{n})$ satisfying Property \ref{eigvectorio}(3) if and only if $u_{0}^{(1)}(\lambda)=0$.
\item[(2)] The respective eigenvectors $(w_{n})_{n \in \mathbf Z}$ for $L_{B,\bq}$ are exponentially decaying sequences and therefore belong to $\ell_{s}^{2}(\bbZ)$ for any $s \geq 0$.
\item[(3)] Equation $u_{0}^{(1)}(\lambda)=u_{0}^{(2)}(\lambda)$  has at least one positive root provided $\bq$ is of type $I_{0}$, equation  $u_{0}^{(2)}(\lambda)=\lambda / \rho_{0}$ has at least one positive root provided $\bq$ is of type $I_{+}$, and equation  $u_{0}^{(1)}(\lambda)=0$ has at least one positive root provided $\bq$ is of type $I_{-}$.
\end{enumerate}
\end{lemma}

\begin{proof}
(1) Let $\bq$ be of type $I_{0}$ and suppose $\lambda \in \sigma_{disc}(L_{B,\bq})$, $\lambda > 0$, with eigenvector $(w_{n})$ that satisfies Property \ref{eigvectorio}(1). We wish to show that $u_{0}^{(1)}(\lambda)=u_{0}^{(2)}(\lambda)$.
Beginning at the eigenvalue equation \eqref{evalproblem}, that is, 
\begin{equation*}
\rho_{n-1}w_{n-1}-\rho_{n+1}w_{n+1}=\lambda w_{n}, \quad n \in \bbZ,
\end{equation*}
and putting $z_{n}=\rho_{n}w_{n}$, we obtain equation \eqref{diffeq1}. 
Notice that Property \ref{eigvectorio} implies that $w_{n} \neq 0$ for any $n$ and hence $z_{n} \neq 0$ for any $n$. Putting $u_{n}=z_{n-1}/z_{n}$, we obtain \eqref{iterate}
from \eqref{diffeq1}. 
%Iterating the first equation above for $n \geq 0$ and the second equation above for $n \leq -1$ we obtain 

Consider the continued fractions \eqref{un1} and \eqref{un2}. 
We claim that $u_{n}^{(1)}(\lambda)=u_{n}$ for every $n \geq 0$ and $u_{n}^{(2)}(\lambda)=u_{n}$ for every $n \leq 0$. This would then imply that $u_{0}^{(1)}(\lambda)=u_{0}^{(2)}(\lambda)$.
%. 

We now give the proof of the fact that the continued fraction defined by $u_n^{(1)}(\lambda)$ matches the $u_{n}$ given by \eqref{iterate} when $n \geq 0$. It follows, from standard facts of continued fractions, see for example \cite[Chapter 2]{JT}, that the odd $k^\text{th}$ truncations $(u_n^{(1)}(\lambda))^{(2k+1)}$ form a monotonically decreasing sequence and the even $k^\text{th}$ truncations $(u_n^{(1)}(\lambda))^{(2k)}$ form a monotonically increasing sequence and $u_n^{(1)}(\lambda)$ is sandwiched in between these. That is, we have, for every $k \geq 1$,
\begin{equation}\label{un1conv}
(u_n^{(1)}(\lambda))^{(2k-2)} \leq (u_n^{(1)}(\lambda))^{(2k)} \leq u_n^{(1)}(\lambda) \leq (u_n^{(1)}(\lambda))^{(2k+1)} \leq (u_n^{(1)}(\lambda))^{(2k-1)}.
\end{equation}
Denote by $u_{n,k}$ the finite continued fraction obtained by iterating the first formula in \eqref{iterate} $k$ times. That is, for every fixed positive integer $k$, $u_{n}=u_{n,k}$ and is given by the formulas
\begin{align*}
u_{n,1} & =\frac{\lambda}{\rho_{n}} +\frac{1}{u_{n+1}}, \quad u_{n,2}=\frac{\lambda}{\rho_{n}} +\cfrac{1}{\frac{\lambda}{\rho_{n+1}}+\frac{1}{u_{n+2}}}, \dots,\\
u_{n}&=u_{n,k}=\frac{\lambda}{\rho_{n}}+\cfrac{1}{\frac{\lambda}{\rho_{n+1}}+ \ddots +\cfrac{1}{\frac{\lambda}{\rho_{n+k-1}}+\frac{1}{u_{n+k}}}}, \quad k \geq 3.
\end{align*}
Since $w_{n} > 0$ for $n \geq 1$ and $w_{0}< 0$, $z_{n}=\rho_{n}w_{n} > 0$ for $n \geq 0$ (recall that for $\bq$ of type $I_{0}$, $\rho_{0}<0$ and $\rho_{n}>0$ for every $n \neq 0$). This then implies that $u_{n} > 0$ for $n \geq 1$. Using this fact, one can directly check that $u_{n}=u_{n,2} \leq (u_n^{(1)}(\lambda))^{(1)}$ and $u_{n}=u_{n,4} \leq (u_n^{(1)}(\lambda))^{(3)}$ and similarly,  $u_{n}=u_{n,3} \geq (u_n^{(1)}(\lambda))^{(2)}$ and $u_{n}=u_{n,5} \geq (u_n^{(1)}(\lambda))^{(4)}$. Proceeding this way, one can directly check that the following holds for every $n \geq 0$ and for fixed $k>0$
\begin{equation*}
(u_n^{(1)}(\lambda))^{(2k)} \leq u_{n,2k+1}=u_n = u_{n,2k+2} \leq (u_n^{(1)}(\lambda))^{(2k+1)}.
\end{equation*}
Taking limits as $k \to \infty$ and using \eqref{un1conv} and the fact that $\lim_{k \to \infty} (u_n^{(1)}(\lambda))^{(k)}=u_n^{(1)}(\lambda)$ one obtains that $u_n^{(1)}(\lambda)=u_{n}$ for $n \geq 0$.

We now prove that $u_{n}^{(2)}(\lambda)=u_{n}$ for $n \leq 0$. The argument is similar to the previous case of $n \geq 0$ and one now needs to keep track of the negative signs in the definition of $u_{n}^{(2)}(\lambda)$ and the fact that $u_{n}<0$ for $n \leq 0$. 
Since $u_{n}^{(2)}(\lambda)$ and its truncations are negative, it follows, from standard facts of continued fractions, see, for example, \cite[Chapter 2]{JT} that the odd $k^\text{th}$ truncations $(u_n^{(2)}(\lambda))^{(2k+1)}$ form a monotonically increasing sequence and the even $k^\text{th}$ truncations $(u_n^{(2)}(\lambda))^{(2k)}$ form a monotonically decreasing sequence and $u_n^{(2)}(\lambda)$ is sandwiched in between these. That is, we have, for every $k \geq 1$,
\begin{equation}\label{un2conv}
(u_n^{(2)}(\lambda))^{(2k-1)} \leq (u_n^{(2)}(\lambda))^{(2k+1)} \leq u_n^{(2)}(\lambda) \leq (u_n^{(2)}(\lambda))^{(2k)} \leq (u_n^{(2)}(\lambda))^{(2k-2)}.
\end{equation}
Denote by $u_{n,k}'$ the finite continued fraction obtained by iterating the second formula in \eqref{iterate} $k$ times. That is, for every $n \leq 0$ and fixed positive integer $k$, $u_{n}=u_{n,k}'$ and is given by the formulas
\begin{align*}
u_{n,1}'&=-\frac{1}{\frac{\lambda}{\rho_{n-1}}-u_{n-1}}, \quad u_{n,2}'=-\cfrac{1}{\frac{\lambda}{\rho_{n-1}}+\cfrac{1}{\frac{\lambda}{\rho_{n-2}}-u_{n-2}}}\, \dots,\\
u_{n}&=u_{n,k}'=-\cfrac{1}{\frac{\lambda}{\rho_{n-1}}+ \cfrac{1}{\ddots+\cfrac{1}{\frac{\lambda}{\rho_{n-k}}-u_{n-k}}}}, \quad k \geq 3.
\end{align*}
Notice that by assumption $u_{n} < 0$ for all $n \leq 0$. One can directly check that $u_{n,1}' > (u_n^{(2)}(\lambda))^{(1)}$ and $u_{n,2}' < (u_n^{(2)}(\lambda))^{(2)}$. Furthermore, the following holds for every $n \leq 0$ and $k \geq 1$,
\begin{equation*}
(u_n^{(2)}(\lambda))^{(2k-1)} \leq u_{n,2k-1}'=u_{n}=u_{n,2k}' \leq (u_n^{(2)}(\lambda))^{(2k)}.
\end{equation*}
Taking limits as $k \to \infty$ and using \eqref{un2conv} and the fact that $\lim_{k \to \infty} (u_n^{(2)}(\lambda))^{(k)}=u_n^{(2)}(\lambda)$ one obtains that $u_n^{(2)}(\lambda)=u_{n}$ for every $n \leq 0$. This proves that $u_{0}^{(1)}(\lambda)=u_{0}^{(2)}(\lambda)$.

Suppose $u_{0}^{(1)}(\lambda)=u_{0}^{(2)}(\lambda)$ for some $\lambda > 0$. We wish to construct an eigenvector $(w_{n})$ that solves the eigenvalue problem \eqref{evalproblem} and satisfies Property \ref{eigvectorio} (1). First define $u_{n}^{(1)}(\lambda)$ and $u_{n}^{(2)}(\lambda)$ as in \eqref{un1} and \eqref{un2} respectively for every $n$, with $\rho_{n}$ given by \eqref{rhonexp}. We now define $u_{n}$ as follows:
\beq\label{undef}
u_{n}=\left\{
	\begin{array}{ll}
		u_{n}^{(1)}(\lambda)  & \mbox{if } n \geq 0, \\
		u_{n}^{(2)}(\lambda) & \mbox{if } n \leq 0.
	\end{array}
\right.
\enq
Note that $u_{n}$ is well defined for all $n \in \mathbb Z$ because of our assumption that $u_{0}^{(1)}(\lambda)=u_{0}^{(2)}(\lambda)$. Furthermore, $u_{n}$ thus defined in \eqref{undef} satisfies \eqref{iterate}. Indeed, one obtains, from \eqref{un1} and \eqref{undef} that for every $n \geq 0$,
 \[u_{n}=u_{n}^{(1)}(\lambda) = \frac{\lambda}{\rho_{n}}+\frac{1}{u^{(1)}_{n+1}(\lambda)} = \frac{\lambda}{\rho_{n}}+\frac{1}{u_{n+1}},\]
where in the second equality above, in the denominator we again used the expression from \eqref{un1} for $u_{n+1}^{(1)}(\lambda)$. Similarly, from \eqref{un2} and \eqref{undef} that for every $n \leq -1$, \[u_{n+1}=u^{(2)}_{n+1}(\lambda)=-\frac{1}{\lambda/ \rho_{n}-u_{n}^{(2)}(\lambda)}=-\frac{1}{\lambda/ \rho_{n}-u_{n}}\]  where, again, in the second equality in the denominator, we used the expression from \eqref{un2} for $u_{n}^{(2)}(\lambda)$. This shows that $u_{n}$ thus defined satisfies \eqref{iterate}. Fix $z_{0}=1$ and for $n \geq 0$ let 
\beq\label{znformulapos}
z_{n}=\frac{z_{0}}{u_{1}u_{2}\ldots u_{n}}, \mbox{ if } n > 0,
\enq and for $n < 0$, we define,
\beq\label{znformulaneg}
z_{n}=z_{0}u_{0}u_{-1}u_{-2}\ldots u_{n+1}, \mbox{ if } n < 0.
\enq
Notice that $z_{n}$ thus defined satisfies $u_n=z_{n-1}/z_n$ for every $n$. Using this one can see that the sequence $(z_{n})_{n \in \bbZ}$ satisfies equation \eqref{diffeq1} because the sequence $(u_{n})_{n \in \bbZ}$ satisfies \eqref{iterate}. We now let $w_{n}=z_{n}/\rho_{n}$ for every $n$ to obtain that the sequence $(w_{n})_{n \in \bbZ}$ satisfies the eigenvalue equation \eqref{evalproblem}. This follows from the fact that $(z_{n})_{n \in \bbZ}$ satisfies the first equation in \eqref{diffeq1}. By construction, since $u_{n}>0$ for $n>0$ and $u_{n}<0$ for $n \leq 0$, one can directly check, using formulas for $z_{n}$ given in equations \eqref{znformulapos}, \eqref{znformulaneg} and the formula $w_{n}=z_{n}/\rho_{n}$ that $(w_{n})$ satisfies Property \ref{eigvectorio} (1).
It follows that $L_{B,\bq}(w_n)_{n\in\bbZ}=\lambda(w_n)_{n\in\bbZ}$, where $(w_{n})$ satisfies Property \ref{eigvectorio} (1) if $u_{0}^{(1)}(\lambda)=u_{0}^{(2)}(\lambda)$. The fact that $(w_n)_{n\in\bbZ} \in \ell^{2}_{s}(\bbZ)$ follows from assertion (2) in the lemma.

%This follows from Lemma \ref{u-lemma} (3).

\noindent{(1P)} \quad Let $\bq$ be of type $I_{+}$ and suppose that $\lambda \in \sigma_{disc}(L_{B,\bq})$, $\lambda > 0$, with eigenvector $(w_{n})$ satisfying Property \ref{eigvectorio}(2). We wish to show that $u_{0}^{(2)}(\lambda)=\lambda /\rho_{0}$. Notice first that in this case $\rho_{1}=0$. Starting with the eigenvalue equation \eqref{evalproblem} and putting $z_{n}=\rho_{n}w_{n}$ we will obtain the equation
\begin{align}\label{diffeq1ip}
z_{n-1}-z_{n+1}&=\frac{\lambda}{\rho_{n}}z_{n}, \quad n \leq -1, \nonumber \\ 
z_{-1}&=\frac{\lambda}{\rho_{0}}z_{0}, \nonumber \\
z_{n}&=0, \quad n \geq 1.
%-z_{3}&=\frac{\lambda}{\rho_{2}}z_{2}, \nonumber \\
%z_{n-1}-z_{n+1}&=\frac{\lambda}{\rho_{n}}z_{n}, n \geq 3. \nonumber  
\end{align}
Now define $u_{n}=z_{n-1}/z_{n}$ for $n < 1$ to obtain the equations
\beq\label{iterateip}
u_n=\frac{\lambda}{\rho_n}+\frac{1}{u_{n+1}}\quad \text{ or }\quad
u_{n+1}=-\frac{1}{\frac{\lambda}{\rho_n}-u_n}, \quad n \leq -1.
\enq
Consider the continued fraction
\begin{equation*}
u_{n+1}^{(2)}(\lambda)=-\cfrac{1}{
\frac{\lambda}{\rho_{n}}+\cfrac{1}{
\frac{\lambda}{\rho_{n-1}}+\cfrac{1}{\frac{\lambda}{\rho_{n-2}}+\dots}}}, \, n=-1,-2,\dots.
\end{equation*}
The proof that $u_{n}^{(2)}(\lambda)=u_{n}$ for $n \leq 0$ is the same as in the case of type $I_{0}$. The second equation in \eqref{diffeq1ip} gives $u_{0}=\lambda / \rho_{0}$ and thus we have, by putting $n=-1$ in the continued fraction above, that $u_{0}^{(2)}(\lambda)=\lambda / \rho_{0}$. 

Now, suppose there exists a positive root $\lambda$ to the equation $u_{0}^{(2)}(\lambda)=\lambda / \rho_{0}$. We wish to construct $(w_{n})$ satisfying Property \ref{eigvectorio} (2) such that $\lambda > 0$ solves the eigenvalue problem \eqref{evalproblem} with eigenvector $(w_{n})$. We first define $u_{n}=u_{n}^{(2)}(\lambda)$ for $n \leq 0$. Notice that by assumption $u_{0}=\lambda/\rho_{0}$. From the definition of the continued fractions $u_{n}^{(2)}(\lambda)$, we can see that the $u_{n}$ thus defined satisfies 
\beq\lb{iteratenew}
u_n=\frac{\lambda}{\rho_n}+\frac{1}{u_{n+1}}\quad \text{ or }\quad
u_{n+1}=-\frac{1}{\frac{\lambda}{\rho_n}-u_n}, \quad n \leq -1.
\enq
Now let $z_{0}=1$ and for $n < 0$, define $z_{n}=z_{0}u_{0}u_{-1}\ldots u_{n+1}$. The $z_{n}$ thus defined satisfies $z_{n-1}/z_{n}=u_{n}$. Also, define $z_{n}=0$ for every $n \geq 1$. From the first equation in \eqref{iteratenew} and using the fact that $u_{0}=\lambda/\rho_{0}$, we obtain the following equations for $z_{n}$,
\begin{align*}
z_{n-1}-z_{n+1}&=\frac{\lambda}{\rho_{n}}z_{n}, \quad n \leq -1, \nonumber \\ 
z_{-1}&=\frac{\lambda}{\rho_{0}}z_{0}, \nonumber \\
z_{n}&=0, \quad n \geq 1.
%-z_{3}&=\frac{\lambda}{\rho_{2}}z_{2}, \nonumber \\
%z_{n-1}-z_{n+1}&=\frac{\lambda}{\rho_{n}}z_{n}, n \geq 3. \nonumber  
\end{align*}
Notice that the third equation above implies that the equation $z_{n-1}-z_{n+1}=\frac{\lambda}{\rho_{n}}z_{n}$ is trivially satisfied for $n > 1$. Using this fact, if we now let $w_{n}=z_{n}/\rho_{n}$ for $n \neq 1$ and $w_{1}=z_{0}/\lambda$, we obtain from the equations above,
\begin{align*}
\rho_{n-1}w_{n-1}-\rho_{n+1}w_{n+1}&=\lambda  w_{n}, \quad n \leq -1 \\
w_{-1}&=\frac{\lambda}{\rho_{-1}} w_{0}\\
w_{1}&=\frac{z_{0}}{\lambda}\\
\rho_{n-1}w_{n-1}-\rho_{n+1}w_{n+1}&=\lambda  w_{n}, \quad n \geq 1.
\end{align*}
The two middle equations above can be rewritten as $\rho_{n-1}w_{n-1}-\rho_{n+1}w_{n+1}=\lambda  w_{n}, \quad n=0$. This is precisely the eigenvalue equation \eqref{evalproblem},
\beq
\rho_{n-1}w_{n-1}-\rho_{n+1}w_{n+1}=\lambda  w_{n}, \quad n \in \mathbb Z,
\enq
where $w_{n}=0$ for $n > 1$ and $w_{n} \neq 0$ when $n \leq 1$. Notice that the $(w_{n})$ thus constructed satisfies Property \ref{eigvectorio} (2). The fact that the eigenfunctions are exponentially decaying follows from part (2) of the Lemma.

\noindent{(1M)} \quad Let $\bq$ be of type $I_{-}$ and suppose that $\lambda \in \sigma_{disc}(L_{B,\bq})$, $\lambda > 0$, with eigenvector $(w_{n})$ satisfying Property \ref{eigvectorio}(3). We need to show that $u_{0}^{(1)}(\lambda)=0$. Starting with the eigenvalue equation \eqref{evalproblem} and putting $z_{n}=\rho_{n}w_{n}$ we will obtain the equation
\begin{align}\label{diffeq1im}
z_{n}&=0, \quad n \leq -1 \\
z_{1}&=-\frac{\lambda}{\rho_{0}}z_{0}, \nonumber \\
z_{n-1}-z_{n+1}&=\frac{\lambda}{\rho_{n}}z_{n}, \quad n \geq 1. \nonumber 
%-z_{3}&=\frac{\lambda}{\rho_{2}}z_{2}, \nonumber \\
%z_{n-1}-z_{n+1}&=\frac{\lambda}{\rho_{n}}z_{n}, n \geq 3. \nonumber  
\end{align}
Now define $u_{n}=z_{n-1}/z_{n}$ for $n > -1$ to obtain the equations
\beq\label{iterateim}
u_n=\frac{\lambda}{\rho_n}+\frac{1}{u_{n+1}}\quad \text{ or }\quad
u_{n+1}=-\frac{1}{\frac{\lambda}{\rho_n}-u_n}, \quad n \geq 0.
\enq
Notice that $u_{0}=z_{-1}/z_{0}=0$. %Iterating the first equation above, we obtain the continued fractions for $u_{n}$ when $n \geq 0$
Consider the continued fraction
\begin{equation*}
u_{n}^{(1)}(\lambda)=\cfrac{1}{
\frac{\lambda}{\rho_{n}}+\cfrac{1}{
\frac{\lambda}{\rho_{n+1}}+\cfrac{1}{\frac{\lambda}{\rho_{n+2}}+\dots}}}, \, n=0,1,2,\dots.
\end{equation*}
By the same proof as in case $I_{0}$, we obtain that $u_{n}^{(1)}(\lambda)=u_{n}$ for every $n \geq 0$. We thus have, by putting $n=0$ in the equation above, that $u_{0}^{(1)}(\lambda)=u_{0}=0$.
%The second equation in \eqref{diffeq1ip} gives $u_{0}=\lambda / \rho_{0}$ and thus we have, by putting $n=-1$ in the equation above, that $u_{0}^{(2)}(\lambda)=\lambda / \rho_{0}$. 

Now, suppose there exists a positive root $\lambda$ to the equation $u_{0}^{(1)}(\lambda)=0$. We wish to construct $(w_{n})$ satisfying Property \ref{eigvectorio} such that $\lambda > 0$ solves the eigenvalue problem \eqref{evalproblem} with eigenvector $(w_{n})$. We first define $u_{n}=u_{n}^{(1)}(\lambda)$ for $n \geq 0$. %Notice that by assumption $u_{0}=\lambda/\rho_{0}$. 
From the definition of the continued fractions, we can see that the $u_{n}$ thus defined satisfies 
\beq\lb{iteratenewim}
u_n=\frac{\lambda}{\rho_n}+\frac{1}{u_{n+1}}\quad \text{ or }\quad
u_{n+1}=-\frac{1}{\frac{\lambda}{\rho_n}-u_n}, \quad n \geq 1.
\enq
Now let $z_{0}=1$ and for $n > 0$, define $z_{n}=\frac{z_{0}}{u_{1}u_{2}\ldots u_{n}}$, $n > 0$. The $z_{n}$ thus defined satisfies $z_{n-1}/z_{n}=u_{n}$. Also, define $z_{n}=0$ for every $n \leq -1$. 
%From the first equation in \eqref{iteratenew} and using the fact that $u_{0}=\lambda/\rho_{0}$, 
We thus obtain the following equations for $z_{n}$,
\begin{align*}
z_{n-1}-z_{n+1}&=\frac{\lambda}{\rho_{n}}z_{n}, \quad n \geq 1, \nonumber \\ 
z_{1}&=-\frac{\lambda}{\rho_{0}}z_{0}, \nonumber \\
z_{n}&=0, \quad n \leq -1.
%-z_{3}&=\frac{\lambda}{\rho_{2}}z_{2}, \nonumber \\
%z_{n-1}-z_{n+1}&=\frac{\lambda}{\rho_{n}}z_{n}, n \geq 3. \nonumber  
\end{align*}
Notice that the third equation above implies that the equation $z_{n-1}-z_{n+1}=\frac{\lambda}{\rho_{n}}z_{n}$ is trivially satisfied for $n < -1$. And the second equation above can be rewritten as  $z_{n-1}-z_{n+1}=\frac{\lambda}{\rho_{n}}z_{n}, \quad n=0$. Using these facts, if we now let $w_{n}=z_{n}/\rho_{n}$ for $n \neq -1$ and $w_{-1}=-z_{0}/\lambda$, we obtain from the equations above,
\begin{align*}
\rho_{n-1}w_{n-1}-\rho_{n+1}w_{n+1}&=\lambda  w_{n}, \quad n \geq 1 \\
w_{1}&=-\frac{\lambda w_{0}}{\rho_{1}}\\
w_{-1}&=-\frac{z_{0}}{\lambda}\\
\rho_{n-1}w_{n-1}-\rho_{n+1}w_{n+1}&=\lambda  w_{n}, \quad n \leq -1.
\end{align*}
The two middle equations above can be rewritten as $\rho_{n-1}w_{n-1}-\rho_{n+1}w_{n+1}=\lambda  w_{n}$ when $n=0$. This is precisely the eigenvalue equation \eqref{evalproblem},
\beq
\rho_{n-1}w_{n-1}-\rho_{n+1}w_{n+1}=\lambda  w_{n}, \quad n \in \mathbb Z,
\enq
where $w_{n}$ satisfies Property \ref{eigvectorio} (3). The fact that the eigenfunctions are exponentially decaying follows from part (2) of the Lemma.

\noindent{(2)} \quad First consider case $I_{0}$. Note that from \eqref{znformulapos}, we have that, 
%if $n\geq 0$, $z_{n}=\frac{z_{0}}{u_{1}u_{2}\ldots u_{n}}$
\begin{equation*}
z_{n}=\frac{z_{0}}{u_{1}u_{2}\ldots u_{n}}, \mbox{ if } n\geq 0.
\end{equation*}
We now use \eqref{u1ineq} to conclude that
\beq\label{znposineq}
|z_{n}| \leq Cq^{n},
\enq
where $C$ is a constant and $0< q< 1$. Note that $q^{n}=e^{n \ln q}=e^{-n\delta}$ for some $\delta > 0$, i.e., we have that if $n \geq 0$,
\beq\label{znexp}
 |z_{n}| \leq Ce^{-n\delta}.
\enq
Notice also, from \eqref{znformulaneg}, we have, 
\begin{equation*}
z_{n}=z_{0}u_{0}u_{1}u_{2}\ldots u_{n+1}, \mbox{ if } n < 0.
\end{equation*}
We now use \eqref{u2ineq} to conclude that \eqref{znposineq} also holds if $n < 0$.
Using arguments similar to that between \eqref{znposineq} and \eqref{znexp} we see that \eqref{znexp} holds if $n < 0$. We thus have that $(z_{n})_{n \in \bbZ} \in \ell^{2}_{s}(\mathbb Z)$ for $s \geq 0$ and since $w_{n}=z_{n} / \rho_{n}$ where $(\rho_{n})_{n \in \bbZ}$ is a bounded sequence with $\lim_{n \to \infty} \rho_{n}=1$, we have that $(w_{n})_{n \in \bbZ} \in \ell^{2}_{s}(\mathbb Z)$ for $s \geq 0$. 

In the case of $I_{+}$, we use the estimates for $z_{n}$ when $n \leq 0$ and set $z_{n}=0$ for $n \geq 1$, i.e., use estimate \eqref{znexp} for $n \leq 0$ and the estimate is also trivially true for $n >0$ thus implying that $(w_{n})_{n \in \bbZ} \in \ell^{2}_{s}(\mathbb Z)$ for $s \geq 0$.

In the case of $I_{-}$, we use the estimates for $z_{n}$ when $n \geq 0$ and set $z_{n}=0$ for $n \leq -1$, i.e., use estimate \eqref{znexp} for $n \geq 0$ and the estimate is also trivially true for $n <0$ thus implying that $(w_{n})_{n \in \bbZ} \in \ell^{2}_{s}(\mathbb Z)$ for $s \geq 0$.

%This follows from Lemma \ref{u-lemma} (4).
(3) We first treat the case $I_{0}$. The fact that $u_{0}^{(1)}(\lambda)=u_{0}^{(2)}(\lambda)$ has a positive root is equivalent to the fact that equation \eqref{eqnlambda} has a positive root $\lambda > 0$. The latter fact follows from \eqref{fglimits}. Indeed, the assertion regarding the two limits in \eqref{fglimits} follow from Lemma \ref{cf} (4) and (5) by replacing $x$ and $(c_{n})$ in equation \eqref{gdef} by $\lambda$ and $(\rho_{n})$ and $(\rho_{-n})$ respectively for $f(\lambda)$ and $g(\lambda)$. The fact that $\rho_{0} < 0$ since $\bq $ is of type $I$ and the fact that by the Van Vleck Theorem, $f,g$ are holomorphic in $\lambda$ provided that $|\arg\lambda| \leq \frac{\pi}{2}-\varepsilon$ together guarantee that \eqref{eqnlambda} has a positive root $\lambda > 0$. 

Next consider the case $I_{+}$. The fact that $u_{0}^{(2)}(\lambda)=\lambda /\rho_{0}$ has a positive root is equivalent to the fact that the equation $g(\lambda)+\lambda/\rho_{0}=0$ has a positive root (recall $u_{0}^{(2)}(\lambda)=-g(\lambda)$). This follows from the facts, as outlined in the case $I_{0}$ above, that $\rho_{0}<0$, $g(\lambda)$ is a holomorphic function provided that $|\arg\lambda| \leq \frac{\pi}{2}-\varepsilon$, and the fact that $g(\lambda)$ is positive for $\lambda >0$ and satisfies the limits $g(\lambda) \to 1$ as $\lambda \to 0^{+}$ and $g(\lambda) \to 0$ as $\lambda \to \infty$.

Next consider the case $I_{-}$. The fact that $u_{0}^{(1)}(\lambda)=0$ has a positive root is equivalent to the fact that the equation $f(\lambda)+\lambda/\rho_{0}=0$ has a positive root. This follows from the facts, as in the case $I_{0}$ and $I_{+}$, that $\rho_{0}<0$, $f(\lambda)$ is a holomorphic function provided that $|\arg\lambda| \leq \frac{\pi}{2}-\varepsilon$, and the fact that $f(\lambda)$ is positive for $\lambda >0$ and satisfies the limits $f(\lambda) \to 1$ as $\lambda \to 0^{+}$ and $f(\lambda) \to 0$ as $\lambda \to \infty$.
\end{proof}
We are ready to present the proof of Theorem \ref{instability}. 
\begin{proof}
%\begin{enumerate}
%\item[(1)] 
(1)\, We begin with the case when $\bq$ is of type $I_{0}$. The fact that equation \eqref{eqio}, $\frac{\lambda}{\rho_{0}}+f(\lambda)+g(\lambda)=0$, has a positive solution is equivalent to the fact that the equation $u_{0}^{(1)}(\lambda)=u_{0}^{(2)}(\lambda)$ has a positive solution $\lambda >0$. This follows from Lemma \ref{lambda-lemma} item (3). Item (1) of Lemma \ref{lambda-lemma} then guarantees that $\lambda > 0$ is an eigenvalue satisfying the eigenvalue equation \eqref{evalproblem} with eigenvector $(w_{n})$ satisfying Property \ref{eigvectorio} if and only if $\lambda > 0$ solves equation $\frac{\lambda}{\rho_{0}}+f(\lambda)+g(\lambda)=0$. The fact that eigenvector $(w_{n})$ forms an exponentially decaying sequence is a consequence of item (2) in Lemma \ref{lambda-lemma} which implies that $(w_{n})_{n \in \bbZ} \in \ell^{2}_{s}(\bbZ)$ for any $s \geq 0$.
\newline
\newline
(2)\, We now consider the case $I_{+}$. The fact that equation \eqref{eqip}, $\frac{\lambda}{\rho_{0}}+g(\lambda)=0$, has a positive solution is equivalent to the fact that the equation $u_{0}^{(2)}(\lambda)=\lambda /\rho_{0}$ has a positive solution $\lambda >0$. This follows from Lemma \ref{lambda-lemma} item (3). Item (1P) of Lemma \ref{lambda-lemma} then guarantees that $\lambda > 0$ is an eigenvalue satisfying the eigenvalue equation \eqref{evalproblem} with eigenvector $(w_{n})$ satisfying Property \ref{eigvectorio} if and only if $\lambda > 0$ solves equation $\frac{\lambda}{\rho_{0}}+g(\lambda)=0$. The fact that eigenvector $(w_{n})$ forms an exponentially decaying sequence is a consequence of item (2) in Lemma \ref{lambda-lemma} which implies that $(w_{n})_{n \in \bbZ} \in \ell^{2}_{s}(\bbZ)$ for any $s \geq 0$.
\newline
\newline
(3)\, We now consider the case $I_{-}$. The fact that equation \eqref{eqim}, $\frac{\lambda}{\rho_{0}}+f(\lambda)=0$, has a positive solution is equivalent to the fact that the equation $u_{0}^{(1)}(\lambda)=0$ has a positive solution $\lambda >0$. This follows from Lemma \ref{lambda-lemma} item (3). Item (1M) of Lemma \ref{lambda-lemma} then guarantees that $\lambda > 0$ is an eigenvalue satisfying the eigenvalue equation \eqref{evalproblem} with eigenvector $(w_{n})$ satisfying Property \ref{eigvectorio} if and only if $\lambda > 0$ solves equation $\frac{\lambda}{\rho_{0}}+f(\lambda)=0$. The fact that eigenvector $(w_{n})$ forms an exponentially decaying sequence is a consequence of item (2) in Lemma \ref{lambda-lemma} which implies that $(w_{n})_{n \in \bbZ} \in \ell^{2}_{s}(\bbZ)$ for any $s \geq 0$.
\end{proof}

Having established an instability argument, we now need to identify when a value of $\mathbf{q}$ can be found of type $I$ for a given $\mathbf{p}$.

\begin{remark}
\label{rem:possibleq}
Let $\mathbf{p}^\perp=(-p_2,p_1)$ where $\mathbf{p}=(p_1,p_2)$. If $\mathbf{q}\in\mathbb{R}^2$ satisfies $\| \mathbf{q}-\frac{3}{4}\mathbf{p}^\perp\|<\frac{1}{4}\|\mathbf{p}\|$, then $\|\mathbf{q}\|<\|\mathbf{p}\|$ and $\|\mathbf{q}\pm \mathbf{p}\|>\|\mathbf{p}\|$. If $\|\mathbf{p}\|>2\sqrt{2}$, then certainly there is a point $\mathbf{q}\in\mathbb{Z}^2$ satisfying the above conditions. Therefore this $\mathbf{q}$ would lead to a subsystem of type $I$ and Theorem \ref{instability} applies.
The proof of this fact is a straightforward geometric exercise analogous with the argument presented in Lemma 4.2 of \cite{WDM1}. 
This defines a $\mathbf{q}$ for all choices of $\mathbf{p}$ satisfying $\|\mathbf{p}\|>2\sqrt{2}$. The small number of exceptions can be checked by hand, leading to the result that an appopriate $\mathbf{q}$ can be found and Theorem \ref{instability} applied in all cases except $\mathbf{p}=(2,1), (1,1)$ and $(1,0)$. Here, $\mathbf{p}=(1,0)$ corresponds to the steady state for the case $\alpha=0$, i.e., the Euler case, described by Arnold \cite{AK98}.\end{remark}

\section{Some auxiliary results on continued fractions}\label{contfracresults}
In this section we collect several simple facts about continued fractions needed in Subsection \ref{ss-cf}. We follow the Appendix in \cite{FH98} and mention \cite{JT} as a general reference. Although the results are not new we have added some arguments not made explicit in \cite{FH98}.

Assume that $(c_{n})_{n\geq 1}$ is a sequence of positive numbers that has a positive limit. For $x>0$ we introduce the function
\begin{equation}\label{gdef}
G(x):=[xc_{1},xc_{2},\ldots]=\cfrac{1}{xc_1+\cfrac{1}{xc_2+\cfrac{1}{xc_3+\ddots}}}
\end{equation}
defined by means of a continued fraction. By changing $x$, when necessary, we can and will assume in what follows that $\lim_{k\to \infty}c_{k}=1$. We note that the continued fraction \eqref{gdef} converges, that is, the limit of the truncated continued fractions
\begin{equation*}
G^{(k)}(x)=  \cfrac{1}{xc_1+\cfrac{1}{xc_2+\cfrac{1}{xc_3+\cfrac{1}{\ddots+\cfrac{1}{xc_{k}}}}}}
\end{equation*}
exists and is positive, that is, $G(x)=\lim_{k \to \infty} G^{k}(x)$. This follows from the Van Vleck Theorem, see \cite[Theorem 4.29]{JT} since $\sum_{k=1}^{\infty}|xc_{k}|=\infty$ by the divergence test. Moreover, the proof of \cite[Theorem 4.29]{JT} based on the Stjeltjes-Vitali Theorem \cite[Theorem 4.30]{JT} yields that the function $G(\cdot)$ is holomorphic for $x \in \bbC$ satisfying $-\frac{\pi}{2}+\varepsilon < \arg(x) < \frac{\pi}{2}+\varepsilon$, for any $\varepsilon >0$.

In addition we will use the notations 
\begin{equation}\label{gndef}
G_{n}(x)=[xc_{n},xc_{n+1},\ldots]=\cfrac{1}{xc_n+\cfrac{1}{xc_{n+1}+\cfrac{1}{xc_{n+2}+\ddots}}}, \quad n=1,2,\ldots
\end{equation}
\begin{equation}\label{ginf}
G_{\infty}(x)=[x,x,\ldots]=\cfrac{1}{x+\cfrac{1}{x+\cfrac{1}{x+\ddots}}},
\end{equation}
and, given positive numbers $a,b>0$, we denote
\begin{equation}\label{f}
F:=F(a,b)=[a,b,a,b,\ldots]=\cfrac{1}{a+\cfrac{1}{b+\cfrac{1}{a+\cfrac{1}{b+\ddots}}}},
\end{equation}
the latter continued fractions also converge by the Van Vleck Theorem.
\begin{lemma}\label{cf}
Assume that $a,b>0, c_{k}>0, \lim_{k \to \infty}c_{k}=1$ and $x>0$. Then the following assertions hold:
\begin{enumerate}\label{lemitem}
\item[(1)] 
\begin{equation}\label{fab}
F(a,b)=\frac{\frac{b}{a}}{\sqrt{(\frac{b}{2})^{2}+\frac{b}{a}}+\frac{b}{2}}
\end{equation}
\item[(2)] If $0<A\leq c_{k} \leq B$ for $k=1,2,\ldots$, then
\begin{equation}\label{gineqab}
\frac{\frac{A}{B}}{\sqrt{(\frac{xA}{2})^{2}+\frac{A}{B}}+\frac{xA}{2}} \leq G(x) \leq \frac{\frac{B}{A}}{\sqrt{(\frac{xB}{2})^{2}+\frac{B}{A}}+\frac{xB}{2}}
\end{equation}
\item[(3)] The limit $\lim_{n\to \infty}G_{n}(x)$ exists and is equal to 
\begin{equation}\label{ginfval}
G_{\infty}(x)=\lim_{n\to \infty}G_{n}(x)=\sqrt{\bigg(\frac{x}{2}\bigg)^{2}+1}-\frac{x}{2}
\end{equation}
\item[(4)] 
\begin{equation}\label{g0}
\lim_{x\to 0^{+}}G(x)=1,
\end{equation}
\item[(5)]
\begin{equation}\label{infg}
\lim_{x\to +\infty}G(x)=0.
\end{equation}
\end{enumerate}
\end{lemma}
\begin{proof}
(1) \sloppy The $k$-th truncated continued fraction for $F(a,b)$ are given by $F^{(2k)}(a,b)=\big[a,b,\ldots,a,b\big]$, $F^{2k+1}(a,b)=\big[a,b,\ldots,a\big]$ and satisfy
\begin{equation*}
F^{(k+2)}(a,b)=\frac{1}{a+\cfrac{1}{b+F^{(k)}(a,b)}}, \quad k=1,2,\ldots .
\end{equation*}
Since the continued fraction $[a,b,\ldots]$ converges, that is, $F^{(k)} \to F$ as $k \to \infty$, we conclude that 
\begin{equation*}
F(a,b)=\frac{1}{a+\cfrac{1}{b+F(a,b)}}, 
\end{equation*}
or $F^{2}(a,b)+bF(a,b)-\frac{b}{a}=0$, yielding \eqref{fab}.

(2) For each $k$-th truncated continued fraction $G^{(k)}(x)=\big[ xc_{1},\ldots,xc_{k}\big]$ we replace the odd-numbered $c_{j}$ by the smaller value $A$ and even-numbered $c_{j}$ by the larger value $B$. Thus, $G^{(k)}(x)$ is majorated by the $k$-th truncation $F^{(k)}(A,B)$ of $\big[A,B,A,B,\ldots\big]$. Passing to the limit as $k \to \infty$ and using (1) yields the second inequality in \eqref{gineqab}. The first inequality follows from $F^{(k)}(B,A) \leq G^{(k)}(x)$.

(3) Formula $G_{\infty}(x)= \sqrt{(\frac{x}{2})^{2}+1}-\frac{x}{2}$ follows from \eqref{fab} with $a=b=x$. It remains to show that the limit $\lim_{n \to \infty} G_{n}(x)$ exists and is equal to $G_{\infty}(x)$. For any $\delta \in (0,1)$ choose $N=N(\delta)$ such that for all $n \geq N$ we have $1-\delta < c_{n} < 1+\delta$. For any $n \geq N$ we apply assertion (2) with $c_{k}$ replaced with $c_{n+k}$, $k=1,2,\ldots$ and $A=1-\delta, B=1+\delta$. This yields
\begin{equation}\label{gnab}
A(x,\delta) \leq G_{n}(x) \leq B(x,\delta), \mbox{ for all } n \geq N,
\end{equation} 
where we introduce the notations
\begin{align}\label{axd}
A(x,\delta)&:= \frac{(1-\delta)/(1+\delta)}{\sqrt{(\frac{x(1-\delta)}{2})^{2}+\frac{1-\delta}{1+\delta}}+\frac{x(1-\delta)}{2}}, \nonumber \\  
B(x,\delta)&:= \frac{(1+\delta)/(1-\delta)}{\sqrt{(\frac{x(1+\delta)}{2})^{2}+\frac{1+\delta}{1-\delta}}+\frac{x(1+\delta)}{2}}. 
\end{align}
We note that $G_{\infty}(x)=\lim_{\delta \to 0} A(x,\delta)= \lim_{\delta \to 0} B(x,\delta), x>0$. For any $\varepsilon > 0$, we fix $\delta=\delta(\varepsilon) \in (0,1)$ such that 
\begin{equation*}
G_{\infty}(x)-\varepsilon < A(x,\delta), G_{\infty}(x)+\varepsilon > B(x,\delta).
\end{equation*}
Then \eqref{gnab} yields $|G_{\infty}(x)-G_{n}(x)| < \varepsilon$ for all $n \geq N(\delta(\varepsilon))$ as claimed.
 
(4) Pick a small $\delta > 0$ to be determined later and choose $N=N(\delta)$ such that \eqref{gnab} holds. Fix an even number $2n > N$ and notice that 
\begin{equation}\label{g1}
G(x)=G_{1}(x)=\big[xc_{1}, xc_{2},\ldots,xc_{2n-1},G_{2n}(x)\big] \leq \big[xc_{1}, xc_{2},\ldots,xc_{2n-1},B(x,\delta)\big],
\end{equation}
where we used that $G_{2n} \leq B(x,\delta)$ by \eqref{gnab}. Clearly, $\lim_{x \to 0} B(x,\delta)=\sqrt{\frac{1+\delta}{1-\delta}}$ yielding
\begin{equation*}
\limsup_{x \to 0} G(x) \leq \bigg[0,\ldots,0,\sqrt{\frac{1+\delta}{1-\delta}}\bigg]=\sqrt{\frac{1+\delta}{1-\delta}}.
\end{equation*}
A similar argument shows that $\liminf_{x \to 0} G(x) \geq \sqrt{\frac{1-\delta}{1+\delta}}$. Passing to the limit as $\delta \to 0$ proves (4).

(5) As before, we arrive at \eqref{g1} and notice that $\lim_{x \to +\infty} B(x,\delta)=0$ by \eqref{axd}. Then 
\begin{equation*}
\lim_{x \to +\infty} \big[xc_{1},xc_{2},\ldots,xc_{2n-1},B(x,\delta)\big]=0
\end{equation*}
yields (5).
\end{proof}

\section{The essential spectrum %of $L_{B}$ 
and the spectral mapping theorem}% for $L_{B}$}
In this section, we follow \cite{LLS} and prove for the linearized $\alpha$-Euler operator that the essential spectrum of the operator $L_{B}$ is the imaginary axis. We also prove the spectral mapping theorem for the group $\{e^{tL_{B}}\}_{t \in \bbR}$ generated by the operator $L_{B}$. 
%This section essentially follows the ideas in \cite{LLS} and proves the analogous results for the linearized $\alpha$-Euler operator. 

First note that $L_{B}$ is the direct sum of operators $L_{B,\bq}$, i.e., $L_B=\oplus_{\bq\in\cQ}L_{B,\bq}$, where $L_{B,\bq}$ is given by
\begin{equation}\label{lbqc}
L_{B,\bq}=(cS-\bar{c}S^*)\diag_{n\in\bbZ}\{1+\gamma_n\},
\end{equation}
with 
\beq\label{c}
c=\frac{\frac12\Gamma(\bq\wedge\bp)}{\|\bp\|^{2}(1+\alpha^{2}\|\bp\|^{2})},
\enq
and $\gamma_{n}$ given by \eqref{gammandef}. We note that in general, if $\Gamma \in \bbC$, then $c$ is a complex number. We thus write $c=|c|e^{i\theta}$ for some $\theta \in [0,2\pi)$. Equation \eqref{lbqc} then becomes,
\begin{equation*}
L_{B,\bq}=|c|(e^{i\theta}S-e^{-i\theta}S^*)\diag_{n\in\bbZ}\{1+\gamma_n\}.
\end{equation*}
\begin{lemma}\label{essspeclbq}
The essential spectrum of the operator $L_{B,\bq}$ is given by
\beq\label{esssplbq}
\sigma_{ess}(L_{B,\bq})=[-2i|c|,2i|c|].
\enq
\end{lemma}
\begin{proof}
We observe that the Fourier transform $\mathbb F: L^{2}(\bbT) \to \ell^{2}(\bbZ):f \mapsto (w_{n})_{n \in \mathbb Z}$ is an isometric isomorphism, where $\mathbb F^{-1}:\ell^{2}(\bbZ) \to L^{2}(\bbT) $ is given by $(w_{n}) \mapsto \sum_{n \in \bbZ}w_{n}e^{inz}$ for $z \in \bbT:=\{z \in \bbC:|z|=1\}$. The operator $e^{i\theta}S-e^{-i\theta}S^*$ acting on $\ell^{2}(\bbZ)$ is similar via $\mathbb F$ to the operator of multiplication by $e^{i\theta}z-e^{-i\theta}\bar{z}$ acting on $L^{2}(\bbT)$, where $z \in \bbT$. That is, \[\mathbb F^{-1}(e^{i\theta}S-e^{-i\theta}S^*)\mathbb F = e^{i\theta}z-e^{-i\theta}\bar{z}.\]
The above equality follows from the observation that \[\mathbb F^{-1}S=z\mathbb F^{-1} \mbox{ and } \mathbb F^{-1}S^{*}=\bar{z}\mathbb F^{-1}.\]
We now use the fact that the spectrum of a multiplication operator on $L^{2}(\bbT)$ is equal to its essential spectrum and is given by the closure of the range of the multiplier. In other words, the spectrum of the operator of multiplication by $e^{i\theta}z-e^{-i\theta}\bar{z}$ on $L^{2}(\bbT)$ is the closure of the range of $e^{i\theta}z-e^{-i\theta}\bar{z}$ as $z \in \bbT$. But this is equal to $[-2i,2i]$. We thus conclude that the essential spectrum of the operator $|c|(e^{i\theta}S-e^{-i\theta}S^*)$ is $[-2i|c|,2i|c|]$. Now, notice that the operator $L_{B,q}$ is a compact perturbation of the operator $|c|(e^{i\theta}S-e^{-i\theta}S^*)$ by the operator $|c|(e^{i\theta}S-e^{-i\theta}S^*)\diag_{n\in\bbZ}\{\gamma_n\}$.  Here, the operator $|c|(e^{i\theta}S-e^{-i\theta}S^*)\diag_{n\in\bbZ}\{\gamma_n\}$ is compact because $|\gamma_{n}| \to 0$ as $|n| \to \infty$. Weyl's theorem \cite[Lemma XIII.4.3]{RS78} allows us to conclude that the essential spectrum of $L_{B,q}$ is the same as the essential spectrum of $|c|(e^{i\theta}S-e^{-i\theta}S^*)$. Thus \eqref{esssplbq} holds. 
\end{proof}
We now prove that the spectrum of $L_{B}$ is exactly the union of the spectra of $L_{B,\bq}$ cf. \cite{LLS}.
\begin{proposition}\label{spectrumlb}
$\sigma(L_{B})=\bigcup_{\bq \in Q}\sigma(L_{B,\bq})$.
\end{proposition}
\begin{proof}
Since $\bigcup_{\bq \in Q}\sigma(L_{B,\bq}) \subset \sigma(L_{B})$ trivially holds, it is enough to show that \[\sigma(L_{B}) \subset  \bigcup_{\bq \in Q}\sigma(L_{B,\bq}).\]
We first split the operator $L_{B}=L^{s}+L^{b}$, where $L^{s}=\oplus_{\|\bq\| \leq \|\bp\|}L_{B,\bq}$ and $L^{b}=\oplus_{\|\bq\| > \|\bp\|}L_{B,\bq}$ correspond to $\bq$ with small and big norms. We have that $\sigma(L_{B})=\sigma(L^{s}) \cup \sigma(L^{b})$, and since $L^{s}$ is the sum of finitely many operators we have that \[\sigma(L_{B})= \bigg(\bigcup_{\|\bq\| \leq \|\bp\|} \sigma(L_{B,\bq})\bigg) \bigcup \sigma(L^{b}).\]
It is thus enough to show that $\sigma(L^{b}) \subset \bigcup_{\|\bq\| > \|\bp\|} \sigma(L_{B,\bq})$. Since $|c| \to \infty$ as $\|\bq\| \to \infty$ (see \eqref{c}), and using the fact that $\sigma_{ess}(L_{B,\bq})=[-2i|c|,2i|c|]$, we see that $i \bbR \subset \bigcup_{\|\bq\| > \|\bp\|} \sigma(L_{B,\bq})$. It therefore suffices to show that $\sigma(L^{b}) \subset i\bbR$. Let us denote \[N_{\bq}^{0}=(e^{i\theta}S-e^{-i\theta}S^*)\]
and \[N_{\bq}=(e^{i\theta}S-e^{-i\theta}S^*)\diag_{n\in\bbZ}\{1+\gamma_n\}.\] Thus $N_{\bq}=N_{\bq}^{0}\diag_{n\in\bbZ}\{1+\gamma_n\}$ and 
$ L_{B,\bq}=|c|N_{\bq}$, i.e., \[ L^{b}=\oplus_{\|\bq\| > \|\bp\|}|c|N_{\bq}.\] 
In order to show that $\sigma(L^{b}) \subset i\bbR$ we show that if $\lambda \notin i \bbR$, then $\lambda $ is in the resolvent set of $L^{b}$. Thus, to prove the proposition,  we need to show that 
\begin{equation}\label{mainclaimlambda}
\mbox{ if} \quad \lambda \notin i\bbR, \quad \mbox{ then} \quad \sup_{\|\bq\| > \|\bp\|}\|\lambda-|c|N_{\bq} \|^{-1} < + \infty. 
\end{equation}
Notice that \[ (\lambda-|c|N_{\bq})^{-1}=\frac{1}{|c|}\bigg(\frac{\lambda}{|c|}-N_{\bq}\bigg)^{-1}.\] 

Notice that $(N_{\bq}^{0})^{*}=-N_{\bq}^{0}$, i.e., $N_{\bq}^{0}$ is a bounded skew self-adjoint operator with $\|N_{\bq}^{0}\| =2$. It's spectrum lies along the imaginary axis and since $\lambda \notin i\bbR$ we have that,
\begin{equation}\label{resnq0}
\bigg\| \bigg(\frac{\lambda}{|c|}-N_{\bq}^{0}\bigg)^{-1}\bigg\|  = \frac{|c|}{|Re(\lambda)|}.
\end{equation}
Also 
\begin{align}
\frac{\lambda}{|c|}-N_{\bq} &= \frac{\lambda}{|c|}-N_{\bq}^{0}-N_{\bq}^{0}\diag_{n\in\bbZ}\{\gamma_n\} \nonumber \\
& = \bigg(\frac{\lambda}{|c|}-N_{\bq}^{0}\bigg)\bigg[ I- \bigg(\frac{\lambda}{|c|}-N_{\bq}^{0}\bigg)^{-1}N_{\bq}^{0}\diag_{n\in\bbZ}\{\gamma_n\} \bigg].
\end{align} 

\textit{Claim:} $|c| \|\diag_{n\in\bbZ}\{\gamma_n\} \| \leq \frac{K (\bp)}{\|\bq\|(1+\alpha^{2}\|\bq\|^{2})}$, where $K (\bp)> 0$ is a constant.  

\noindent \textit{Proof of Claim:} \quad Using the definition of $\gamma_{n}$ (see \eqref{gammandef}) and $c$ (see \eqref{c}) we have, 
\[ |c\gamma_{n}| = \frac{|\Gamma| |\bq \wedge \bp|}{2\|\bq+n\bp\|^{2}(1+\alpha^{2}\|\bq+n \bp\|^{2})}.\]
Now use the fact that $\bq \wedge \bp = (\bq+n\bp) \wedge \bp$ and the fact that $|\bq \wedge \bp|=|\bq \cdot \bp^{\perp}|$ and the Cauchy-Schwarz inequality to see that $ |\bq \wedge \bp| = |(\bq +n\bp) \wedge \bp| \leq \|\bq +n\bp\| \|\bp \| $. This then implies that, 
\[ |c\gamma_{n}| \leq \frac{K(\bp)}{\|\bq+n\bp\|(1+\alpha^{2}\|\bq+n\bp\|^{2})}.\]
We thus have that,
\[ |c| \|\diag_{n\in\bbZ}\{\gamma_n\} \| \leq |c| \sup_{n} |\gamma_{n}| \leq \frac{K(\bp)}{\|\bq\|(1+\alpha^{2}\|\bq\|^{2})},\]
which finishes the proof of the Claim. 

Now choose $\|\bq_{0}\| > \|\bp\|$ so that for all $\|\bq\| \geq \|\bq_{0}\|$, the inequality
\begin{equation}\label{ccd}
\frac{2K(\bp)}{|Re(\lambda)\|\bq\|(1+\alpha^{2}\|\bq\|^{2})|} \leq \frac{1}{2}
\end{equation}
holds. We stress that $\bq_{0}$ depends on $Re(\lambda)$ but does not depend on $Im(\lambda)$. Denote $Q_{s}:=\{ \bq \in Q: \|\bq\| \in [\|\bp\|,\|\bq_{0}\|]\}$ and $Q_{b}:=\{ \bq \in Q: \|\bq\| \geq \|\bq_{0}\|\}$. If $\bq \in Q_{b}$, using \eqref{ccd}, and the fact that $\|N_{\bq}^{0}\| =2$, we have,
\begin{align*} 
\bigg\|\bigg(\frac{\lambda}{|c|}-N_{\bq}^{0}\bigg)^{-1}N_{\bq}^{0}\diag_{n\in\bbZ}\{\gamma_n\}\bigg\| &\leq \frac{2|c|\|\diag_{n\in\bbZ}\{\gamma_n\}\|}{|Re(\lambda)|} \\&\leq \frac{2K(\bp)}{|Re(\lambda)\|\bq\|(1+\alpha^{2}\|\bq\|^{2})|} \leq \frac{1}{2}.
\end{align*}
This proves that as long as $\bq \in Q_{b}$, the operator $\bigg[ I- \bigg(\frac{\lambda}{|c|}-N_{\bq}^{0}\bigg)^{-1}N_{\bq}^{0}\diag_{n\in\bbZ}\{\gamma_n\} \bigg]$ is invertible and 
\[ \bigg\| \bigg[ I- \bigg(\frac{\lambda}{|c|}-N_{\bq}^{0}\bigg)^{-1}N_{\bq}^{0}\diag_{n\in\bbZ}\{\gamma_n\} \bigg]\bigg\| \leq 2.\]
Therefore, as long as $\bq \in Q_{b}$, we have that 
\begin{align*}
 \|(\lambda-|c|N_{\bq})^{-1}\|=\frac{1}{|c|}\bigg\|\bigg(\frac{\lambda}{|c|}-N_{\bq}\bigg)^{-1}\bigg\| \leq \frac{1}{|c|}\frac{|c|}{|Re(\lambda)|}2 = \frac{2}{|Re(\lambda)|}.
\end{align*}
Thus, 

\begin{equation}\label{qbineq}
 \sup_{\bq \in Q_{b}} \|(\lambda-|c|N_{\bq})^{-1}\| \leq \frac{2}{|Re(\lambda)|}.
\end{equation}

To finish the proof, we note that the set $Q_{s}$ is finite and since $(\lambda-|c|N_{\bq})^{-1}$ is a bounded linear operator for every $\bq \in Q_{s}$, it follows that $\oplus_{\bq \in Q_{s}}\|(\lambda-|c|N_{\bq})^{-1}\| $ is also a bounded linear operator, where, if $\lambda=Re(\lambda)+iIm(\lambda)$, with $Re(\lambda) \neq 0$, then the resolvent operator grows as $O(1/(|Im(\lambda)|)$ as $|Im(\lambda)| \to \infty$. We have that,
\begin{equation}\label{qsineq}
 \sup_{\bq \in Q_{s}} \|(\lambda-|c|N_{\bq})^{-1}\| < + \infty.
\end{equation}
Since $\{\bq:\|\bq\| > \|\bp\|\}=Q_{s} \cup Q_{b}$, equations \eqref{qbineq}, \eqref{qsineq} show that \eqref{mainclaimlambda} holds. This proves the proposition.
\end{proof}

\begin{proposition}\label{essentialspectrumlb}
$\sigma_{ess}(L_{B})=i\bbR$ and $\sigma_{p}(L_{B}) \backslash i \bbR = \bigcup_{\|\bq\| \leq \|\bp\|} (\sigma_{p}(L_{B,\bq})\backslash i \bbR)$ is a bounded set with accumulation points only on $i\bbR$.
\end{proposition}
\begin{proof}
The facts that $|c| \to \infty$ as $\|\bq\| \to \infty$ and \eqref{esssplbq}, together with the fact that $\bigcup_{\bq \in Q}\sigma_{ess}(L_{B,\bq}) \subset \sigma_{ess}(L_{B})$ imply that $i\bbR \subset \sigma_{ess}(L_{B})$. It is thus enough to prove that $\sigma_{ess}(L_{B}) \subset i\bbR$.
We have, \[\sigma_{ess}(L_{B})=\bigcup_{\|\bq\| \leq \|\bp\|} \sigma_{ess}(L_{B,\bq}) \bigcup \sigma_{ess}(L^{b}).\] 
Notice that, since $\oplus_{\|\bq\| \leq \|\bp\|} (L_{B,\bq})$ is a sum of finitely many bounded linear operators and using \eqref{esssplbq}, we have that \[ \bigcup_{\|\bq\| \leq \|\bp\|} \sigma_{ess}(L_{B,\bq}) \subset i \bbR.\]
From the proof of Proposition \ref{spectrumlb}, see Equation \eqref{mainclaimlambda}, we know that $\sigma(L^{b}) \subset i\bbR$, i.e., $L^{b}$ does not have points in the spectrum with non zero imaginary values. Thus, \[\sigma_{ess}(L^{b}) \subset \sigma(L^{b}) \subset i\bbR.\] This proves that $\sigma_{ess}(L_{B}) \subset i \bbR$. The second statement of the Proposition follows from the above and from the fact that $\bigcup_{\|\bq\| \leq \|\bp\|} L_{B,\bq}$ is a finite sum of bounded linear operators. 
%We know that $L^{b}$ does not have points in the spectrum with non zero imaginary values and the facts that $|c| \to \infty$ as $\|\bq\| \to \infty$ and \eqref{esssplbq} imply that $i\bbR \subset \sigma_{ess}(L_{B})$.
\end{proof}
We now prove the spectral mapping theorem for the operator $L_{B}$.
\begin{proposition}\label{spmap}
The spectral mapping property,
\[ \sigma(e^{tL_{B}})=e^{t\sigma(L_{B})}, \quad t \neq 0,\]
holds for the operator $L_{B}$.
\end{proposition}
\begin{proof}
\sloppy
We know from Proposition \ref{essentialspectrumlb}, that the essential spectrum of $L_{B}$ satisfies $\sigma_{ess}(L_{B})=i\bbR$. This tells us that $e^{t \sigma_{ess}(L_{B})}=e^{i\bbR}=\{z \in \bbC:|z|=1\}$. Since $e^{t \sigma_{ess}(L_{B})} \subseteq \sigma(e^{tL_{B}})$ for any semigroup, we see that $\{z \in \bbC:|z|=1\} \subseteq \sigma(e^{tL_{B}})$. We want to show that $ \sigma_{ess}(e^{tL_{B}}) \subset \{z \in \bbC:|z|=1\}$. We use a general Gearhart-Pruss spectral mapping theorem for Hilbert spaces, see \cite[Th.2, p.268]{LLS}. On a Hilbert space, $\sigma(e^{tL_{B}})$, $t \neq 0$, is the set of points $e^{\lambda t}$ such that either $\mu_{n}=\lambda+2\pi n / t$ belongs to $\sigma(L_{B})$ for all $n \in \bbZ$ or the sequence $\{\|R(\mu_{n},L_{B})\|\}_{n \in \bbZ}$ is unbounded. Suppose $\sigma_{ess}(e^{tL_{B}}) \not\subset \{z \in \bbC:|z|=1\}$. Then, there exists $e^{t\lambda}$ such that $\lambda \notin i\bbR$ and either $\mu_{n}=\lambda+2\pi n / t \in \sigma_{ess}(L_{B})$ for all $n \in \bbZ$ or the sequence $\{\|R(\mu_{n},L_{B})\|\}_{n \in \bbZ}$ is unbounded. The first outcome is precluded by the fact that $\sigma_{ess}(L_{B})=i\bbR$. So if $e^{t\lambda} \notin \{z \in \bbC:|z|=1\}$ and $e^{t\lambda} \in \sigma_{ess}(e^{tL_{B}})$ then we must have that $\sup_{y \in \bbR}\|R(Re(\lambda)+iy,L_{B})\| = +\infty$. But this is impossible because, as we prove below that for each $\lambda \notin i \bbR$, $\sup_{y \in \bbR}\|R(Re(\lambda)+iy,L_{B})\| < + \infty$. So it remains to establish the following fact. 

Claim: Assume $\{Re(\lambda)+iy:y \in \bbR\} \cap \sigma(L_{B}) = \varnothing, Re(\lambda) > 0$, then $\sup_{y \in \bbR}\|R(Re(\lambda)+iy,L_{B})\| < \infty$.

Let $\lambda \notin i\bbR$ as in the proof of Proposition \ref{spectrumlb} and fix $Re(\lambda)$. Since $Q_{s}$ is a finite set, the operator $\|R(\lambda,\oplus_{\bq \in Q_{s}}L_{B,\bq})\|$ is a bounded linear operator such that the norm of its resolvent decays as $O(1/(|Im(\lambda)|)$ as $|Im(\lambda)| \to \infty$ and \eqref{qsineq} holds, i.e., one has $\|R(\lambda,\oplus_{\bq \in Q_{s}}L_{B,\bq})\| \leq C$. One also has that if $\bq \in Q_{b}$, then \eqref{qbineq} holds, i.e., the norm of the resolvent operator $\|R(\lambda,\oplus_{\bq \in Q_{b}}L_{B,\bq})\| \leq C/ |Im(\lambda)|$. These two facts above can be combined to give 
\begin{equation}\label{resolvest}
\| (\lambda-L_{B})^{-1}\| =O(1) \quad \mbox{as} \quad |Im(\lambda)| \to \infty.
\end{equation}
By estimate \eqref{resolvest}, we know that if $Re(\lambda) \neq 0$, then $e^{t\lambda}$ is not in the spectrum of $e^{tL_{B}}$. This shows that the essential spectrum of $e^{tL_{B}}$, $\sigma_{ess}(e^{tL_{B}})$, is contained in the unit circle. One also knows that the spectral mapping property always holds for the point spectrum. One can combine these facts to obtain the result.
%One can then use a general Gearhart-Pruss spectral mapping theorem for Hilbert spaces to conclude the result, see \cite[Th.2, pg 268]{LLS}. On a Hilbert space, $\sigma(e^{tL_{B}})$, $t \neq 0$, is the set of points $e^{\lambda t}$ such that either $\mu_{n}=\lambda+2\pi n / t$ belongs to $\sigma(L_{B})$ or the sequence $\{\|R(\mu_{n},L_{B})\|\}_{n \in \bbZ}$ is unbounded. 
\end{proof}

\section{Concluding comments}

The main result of the present paper, Theorem \ref{instability}, states that for steady states (which are a function of the vector $\mathbf p$), that have a point $\mathbf q$ of type $I$, that is, the set $\{\mathbf q +n \mathbf p: n \in \mathbb Z\}$ has one point inside the open disc of radius $\mathbf p$, (see Subsection \ref{B-S}, Remark \ref{classf}  for a precise definition of point of type I) are linearly unstable. The existence of an unstable eigenvalue is equivalent to the existence of a positive root to an equation involving continued fractions (equations \eqref{eqio}, \eqref{eqip}, \eqref{eqim}, respectively, for points of type $I_{0}$, $I_{+}$ and $I_{-}$). We are able to provide a list of additional properties that the respective eigenvectors satisfy. In Section 4, we also characterized the essential spectrum of the operator $L_{B}$ and proved a spectral mapping theorem for the group generated by $L_{B}$. Moving forward, proving linear instability, as done in the current paper, can be seen as a first step to prove full nonlinear instability. In \cite{GHS07}, the authors are able to characterize the nonlinear growth rate of the solution in terms of the largest real eigenvalue of the linearized operator. This allows them to study the nonlinear instability of interfacial fluid motions, examples of which include vortex sheets with surface tension and Hele-Shaw flows. In \cite{FV11}, the authors prove nonlinear instability and ill-posedness of the magneto-geostrophic equations by first proving that the linearized operator has an unstable eigenvalue using continued fractions techniques. Similarly, in \cite{FGSV12}, the authors study the ill-posedness of a nonlinear singular porous media equation by studying first the instability of the linearized operator using continued fractions. In \cite{KS14}, the authors show that the gradient of the vorticity to the 2D nonlinear Euler equation has double exponential growth rate. A related question is to show single exponential growth in the case of the 2D $\alpha$-Euler equations. 

\section{Acknowledgements}
We thank the referees for their valuable comments which improved the exposition of our paper. 

\section{Appendix}
The purpose of this Appendix is to collect some proofs of results used in the main body of the text.
\begin{lemma}\label{eqfmode}
Equation \eqref{vortalpha} holds if and only if $\omega_{\bk}$ satisfies equation \eqref{DEEalpha} for every $\bk \neq 0$.
\end{lemma}
\begin{proof}
Using the facts that $v_1=\frac{\partial\phi}{\partial y}$ and $v_2=-\frac{\partial\phi}{\partial x}$, one can rewrite equation \eqref{vortalpha} as 
\beq\label{EEphi}
\frac{\partial\omega}{\partial t}=-\frac{\partial\phi}{\partial y}\frac{\partial\omega}{\partial x}+\frac{\partial\phi}{\partial x}\frac{\partial\omega}{\partial y}.
\enq
Using \eqref{ophif}, we see that, 
\[ \frac{\partial \phi}{\partial x} = \sum_{\bq\in\bbZ^2\setminus\{0\}}  \frac{ik_{1}\omega_{\bk}e^{i\bk\cdot\bx}}{||\bk||^{2}(1+\alpha^{2}||\bk||^{2})}, \quad \frac{\partial \phi}{\partial y} = \sum_{\bq\in\bbZ^2\setminus\{0\}}  \frac{ik_{2}\omega_{\bk}e^{i\bk\cdot\bx}}{||\bk||^{2}(1+\alpha^{2}||\bk||^{2})}.\]
%$\o=\Delta\psi$ and $\psi(\bx)=-\sum_{\bk\in\bbZ^2\setminus\{0\}}\|\bk\|^{-2}\o_\bk e^{i\bk\cdot\bx}$. 
Equation \eqref{EEphi}  then reads, in terms of the Fourier series,
\beq\label{EEfs}\begin{split}
\frac{\partial\omega}{\partial t}&=-\bigg( \sum_{\bk\in\bbZ^2\setminus\{0\}}\frac{ik_{2}\omega_{\bk}}{||\bk||^{2}(1+\alpha^{2}||\bk||^{2})}  e^{i\bk\cdot\bx}\bigg) \bigg(\sum_{\bk\in\bbZ^2\setminus\{0\}} ik_1\omega_\bk e^{i\bk\cdot\bx}\bigg)\\&+\bigg( \sum_{\bk\in\bbZ^2\setminus\{0\}}\frac{ik_{1}\omega_{\bk}}{||\bk||^{2}(1+\alpha^{2}||\bk||^{2})} e^{i\bk\cdot\bx}\bigg) \bigg(\sum_{\bk\in\bbZ^2\setminus\{0\}} ik_2\omega_\bk e^{i\bk\cdot\bx}\bigg).
\end{split}\enq
%\beq\label{EEfs}\begin{split}
%\frac{\partial\omega}{\partial t}&=-\bigg( \sum_{\bk\in\bbZ^2\setminus\{0\}}\frac{ik_{2}\omega_{\bk}}{||\bk||^{2}(1+\alpha^{2}||\bk||^{2})}\omega_\bk e^{i\bk\cdot\bx}\bigg) \bigg(\sum_{\bk\in\bbZ^2\setminus\{0\}} ik_1\omega_\bk e^{i\bk\cdot\bx}\bigg)\\&+\bigg( \sum_{\bk\in\bbZ^2\setminus\{0\}}\frac{ik_{1}\omega_{\bk}}{||\bk||^{2}(1+\alpha^{2}||\bk||^{2})}\omega_\bk e^{i\bk\cdot\bx}\bigg) \bigg(\sum_{\bk\in\bbZ^2\setminus\{0\}} ik_2\omega_\bk e^{i\bk\cdot\bx}\bigg).
%\end{split}\enq
%where we notice the sign change due to the presence of the term $ik_1$ and $ik_2$ in the respective factors.
Using the identity
\[\big(\sum_\bn a_\bn e^{i\bn\cdot\bx}\big)\big(\sum_\bl b_\bl e^{i\bl \cdot\bx}\big)
=\sum_\bk\big(\sum_{\bq}
a_{\bq} b_{\bk-\bq} e^{i\bk\cdot\bx}\big)
%=\sum_\bk
%\big(\sum_{\bq}
%a_{\bk-\bq} b_\bq e^{i\bk\cdot\bx}\big)
\]
first for  $a_\bn=n_2\|\bn\|^{-2}(1+\alpha^{2}\|\bn\|^{2})^{-1}\omega_\bn$, $b_\bl=l_1\omega_\bl$ and then for  $a_\bn=n_1\|\bn\|^{-2}(1+\alpha^{2}\|\bn\|^{2})^{-1}\omega_\bn$, $b_\bl=l_2\omega_\bl$, equation \eqref{EEfs} is seen to be
\beq\lb{A1}
\frac{\partial\omega}{\partial t}=  \sum_{\bk\in\bbZ^2\setminus\{0\}} \sum_{\bq\in\bbZ^2\setminus\{0\}} \frac{q_{2}(k_{1}-q_{1})-q_{1}(k_{2}-q_{2})}{\|\bq\|^{2}(1+\alpha^{2}\|\bq\|^{2})}\omega_{\bk-\bq}\omega_\bq e^{i\bk\cdot\bx}.
\enq

Alternatively, using the identity
\[\big(\sum_\bn a_\bn e^{i\bn\cdot\bx}\big)\big(\sum_\bl b_\bl e^{i\bl \cdot\bx}\big)
=\sum_\bk\big(\sum_{\bq}
a_{\bk-\bq} b_{\bq} e^{i\bk\cdot\bx}\big)
%=\sum_\bk
%\big(\sum_{\bq}
%a_{\bk-\bq} b_\bq e^{i\bk\cdot\bx}\big)
\]
first for  $a_\bn=n_2\|\bn\|^{-2}(1+\alpha^{2}\|\bn\|^{2})^{-1}\omega_\bn$, $b_\bl=l_1\omega_\bl$ and then for  $a_\bn=n_1\|\bn\|^{-2}(1+\alpha^{2}\|\bn\|^{2})^{-1}\omega_\bn$, $b_\bl=l_2\omega_\bl$, equation \eqref{EEfs} is seen to be
\beq\lb{A2}
\frac{\partial\omega}{\partial t}=  \sum_{\bk\in\bbZ^2\setminus\{0\}} \sum_{\bq\in\bbZ^2\setminus\{0\}} \frac{q_{1}(k_{2}-q_{2})-q_{2}(k_{1}-q_{1})}{\|\bk-\bq\|^{2}(1+\alpha^{2}\|\bk-\bq\|^{2})}\omega_{\bk-\bq}\omega_\bq e^{i\bk\cdot\bx}.
\enq
Noticing that $\frac{\partial\omega}{\partial t} = \sum_{\bk\in\bbZ^2\setminus\{0\}} \frac{d \omega_{\bk}}{ dt} e^{i\bk\cdot\bx}$ and taking the average of \eqref{A1} and \eqref{A2} we obtain that \eqref{DEEalpha} for each mode $\omega_{\bk}$ of $\omega$ holds if and only if \eqref{vortalpha} holds.
\end{proof}
We now prove that the unidirectional flow given by \eqref{dfnBS} and \eqref{steadystuni} is a steady state.
\begin{lemma}\label{steadyst}
A unidirectional flow given by the vorticity equations \eqref{dfnBS} and \eqref{steadystuni}
is a steady state solution of the $\alpha$-Euler equation \eqref{vortalpha} on the torus $\bbT^{2}$.
\end{lemma}
%Proof of Lemma \ref{steadyst}.
\begin{proof}
For every $\bk \neq 0$ one needs to check that the right hand side of \eqref{DEEalpha} is zero, where the Fourier coefficients of $\omega^{0}_{\bk}$ are given by \eqref{steadystuni}. Since $\omega^{0}_{\bq}$ is nonzero only when $\bq =\pm \bp$, the right hand side of\eqref{steadystuni} reduces to 
\[ \beta(\bk-\bp,\bp)\omega^{0}_{\bk-\bp}\omega^{0}_{\bp}+ \beta(\bk+\bp,\bp)\omega^{0}_{\bk+\bp}\omega^{0}_{\bp}.\]
Now using the fact that $\omega^{0}_{\bk-\bp}$ is nonzero only when $\bk-\bp=\pm \bp$ and $\omega^{0}_{\bk+\bp}$ is nonzero only when $\bk+\bp=\pm \bp$ and using \eqref{steadystuni}, the above equation reduces to 
\[ \frac{1}{4}\big(\beta(\bp,\bp)\Gamma^{2} + \beta(-\bp,\bp)\overline{\Gamma} \Gamma + \beta(\bp,-\bp) \Gamma \overline{\Gamma} + \beta(-\bp,-\bp) {\overline{\Gamma}}^{2}\big) , \] which is zero because $\beta(\bp,\pm \bp)=0$ and $\beta(\pm \bp, \bp)=-\beta(\bp,\pm \bp)$. 
\end{proof}

Derivation of Equation \eqref{dfnLB}:

We briefly indicate how to obtain equation \eqref{dfnLB}. Linearizing the right hand side of \eqref{DEEalpha} about the steady state \eqref{dfnBS} reduces the right hand side of  \eqref{DEEalpha} to 
\begin{equation}\label{lin}
\sum_{\bq\in\bbZ^2\setminus\{0\}}\beta(\bk-\bq,\bq)\omega^{0}_{\bk-\bq}\omega_\bq + \sum_{\bq\in\bbZ^2\setminus\{0\}}\beta(\bk-\bq,\bq)\omega_{\bk-\bq}\omega^{0}_\bq,
\end{equation}
where in the first sum, $\omega^{0}_{\bk-\bq}=\Gamma /2 $ if $\bk-\bq=\bp$, i.e., if $\bq=\bk-\bp$ and $\omega^{0}_{\bk-\bq}=\bar{\Gamma} /2 $ if $\bk-\bq=-\bp$, i.e., if $\bq=\bk+\bp$ and zero otherwise and in the second sum, $\omega^{0}_{\bq}=\Gamma/2$ if $\bq=\bp$ and $\omega^{0}_{\bq}=\bar{\Gamma}/2$ if $\bq=-\bp$ and zero otherwise. Using these in \eqref{lin}, we see that \eqref{lin} reduces to,
\begin{equation*}
\beta(\bp,\bk-\bp)\frac{\Gamma}{2}\omega_{\bk-\bp}+ \beta(-\bp,\bk+\bp)\frac{\bar{\Gamma}}{2}\omega_{\bk+\bp} + \beta(\bk-\bp,\bp)\frac{\Gamma}{2}\omega_{\bk-\bp} + \beta(\bk+\bp,-\bp)\frac{\bar{\Gamma}}{2}\omega_{\bk+\bp}.
\end{equation*}
Now use the facts that if $\bp \neq \bq$, then $\beta(\bp,\bq)=\beta(\bq,\bp)$ and $\beta(-\bp,\bq)=-\beta(\bp,\bq)$ in the above equation to get \eqref{dfnLB}.\newline

We now give the proof of Lemma \ref{symev}.

\begin{proof}
Recall \eqref{dfnrho} and the assumption that $\Gamma \in \bbR$. Note that $L_{B,\bq}=(S-S^*)\diag_{n\in\bbZ}\{\rho_n\}$, where \[(S-S^{*})^{*}=S^{*}-S = -(S-S^{*}).\] We thus have that, 
\begin{align*} 
\sigma(L_{B,\bq}^{*})\backslash \{0\} &= \sigma (\rho_{n}(S-S^{*})^{*}) \backslash \{0\}= -\sigma (\rho_{n}(S-S^{*})) \backslash \{0\} \\&= -\sigma ((S-S^{*})\rho_{n}) \backslash \{0\} = - \sigma(L_{B,\bq}) \backslash \{0\}. 
\end{align*}
Thus $\sigma(L_{B,\bq}) \backslash \{0\} = \overline{\sigma(L_{B,\bq}^{*})} \backslash \{0\}=-\overline{\sigma(L_{B,\bq})} \backslash \{0\}$. Thus the eigenvalues are symmetric about the imaginary axes. 

The fact that the eigenvalues are symmetric about the real axes can be proved as follows. The fact that if $\lambda$ is an eigenvalue then $\overline{\lambda}$ is also an eigenvalue is a consequence of the fact that $\overline{L_{B,\bq} \bv} = L_{B,\bq} \overline{\bv}$ for any $\bv \in \ell^{2}(\bbZ)$. From this it follows that if $\lambda$ is an eigenvalue with eigenvector $\bv$, then $\overline{\lambda}$ is an eigenvalue with eigenvector $\overline{\bv}$. This proves the Lemma. 

Additionally, one can also prove the fact that if $\lambda$ is an eigenvalue then $-\lambda$ is also an eigenvalue. Let $\hat{J}$ be an operator on $\ell^{2}(\bbZ)$ defined by $(\o_{n}) \mapsto ((-1)^{n}\o_{n})$ and notice that $\hat{J}S=-S\hat{J}$ and $\hat{J}S^{*}=-S^{*}\hat{J}$ and $\hat{J}^{2}=I$. Thus, 
\[ \hat{J}L_{B,\bq} \hat{J} = \hat{J} ((S-S^{*})\diag_{n\in\bbZ}\{\rho_n\}) \hat{J}=-L_{B,\bq}.\] Thus,
\[ 
\sigma(L_{B,\bq}) = \sigma(L_{B,\bq}\hat{J}\hat{J})=\sigma(\hat{J}L_{B,\bq}\hat{J})=-\sigma(L_{B,\bq}),\]
which concludes the proof. We used Lemma \ref{spectrumcommute} in the last part of the proof.
\end{proof}

%%%%%%%%%%%%%%%%%%%%%
%%%%%%%%%%%%%%%%%%%%%%%%%%%%

%%%%%%%%%%%%%%%%%%%%%%%%
%%%%%%%%%%%%%%%%%%%%%%%%%%%%

%%%%%%%%%%%%%%%%%%%%%%%%

\end{document}